\documentclass[a4paper,twoside,11pt]{amsart}

\usepackage{amsthm,amsmath,amssymb, amsfonts,stmaryrd,multicol,dsfont}
\usepackage[usenames,dvipsnames,svgnames,table]{xcolor}
\usepackage{tikz}
\usepackage{subfig}
\usepackage{mathtools}
\usepackage{faktor}

\usepackage{hyperref}
\hypersetup{colorlinks=true, breaklinks=true, urlcolor= blue, linkcolor= blue, citecolor=orange,linktocpage,pdftitle=Spectrum of random hyperbolic surfaces of high genus}

\usepackage[left=3cm,right=3cm,top=3cm,bottom=3cm]{geometry}

\theoremstyle{theorem}
\newtheorem{theo}{Theorem}
\newtheorem{prop}[theo]{Proposition}
\newtheorem{coro}[theo]{Corollary}
\newtheorem{lemm}[theo]{Lemma}
\newtheorem{conj}[theo]{Conjecture}
\theoremstyle{definition}

\DeclarePairedDelimiter{\abs}{|}{|}
\DeclarePairedDelimiter{\brac}{[}{]}
\DeclarePairedDelimiter{\paren}{(}{)}
\DeclarePairedDelimiter{\acc}{\{}{\}}

\newcommand{\R}{\mathbb{R}}

\renewcommand{\epsilon}{\varepsilon}
\renewcommand{\d}{\; \text{d}}

\newcommand{\C}{\mathbb{C}}
\renewcommand{\Im}{\mathrm{Im} \;}

\DeclareMathOperator{\volhyp}{\mu}
\newcommand{\dvolhyp}{\; \mathrm{d} \mu}
\DeclareMathOperator{\injrad}{InjRad}

\DeclareMathOperator{\sinc}{sinc}

\newcommand{\id}{\mathrm{id}}
\newcommand{\dist}{\mathrm{d}}

\newcommand{\volwp}{\mathrm{Vol}^{\mathrm{WP}}_{g}}
\newcommand{\omegawp}{\omega^{\mathrm{WP}}_g}
\newcommand{\Pwp}{\mathbb{P}^{\mathrm{WP}}_{g}}
\newcommand{\Ewp}{\mathbb{E}^{\mathrm{WP}}_{g}}

\newcommand{\counting}[3]{\mathrm{N}^{#1}_{#2} \paren*{#3}}
\newcommand{\1}[1]{\mathds{1}_{#1}}
\newcommand{\unreg}[1]{\tilde{\mathds{1}}_{#1}}
\renewcommand{\O}[1]{\mathcal{O} \left( #1 \right)}
\newcommand{\integraltermH}{\frac{1}{4 \pi} \int_{a}^{b} \tanh \paren*{\pi
    \sqrt{\lambda - \frac 1 4}} \d \lambda}
\newcommand{\integralterm}{\frac{1}{4 \pi} \int_{\frac 1 4}^{+ \infty} \1{[a,b]}(\lambda) \tanh \paren*{\pi
    \sqrt{\lambda - \frac 1 4}} \d \lambda}
\newcommand{\decay}{s}
\newcommand{\tminB}{\frac{1}{200}}
\newcommand{\tminH}{\frac{1}{10}}
\newcommand{\rmax}{3}
\newcommand{\rhocut}{\rho_{\mathrm{cut}}}
\newcommand{\jcut}{j_{\mathrm{cut}}}
\newcommand{\bulkdist}{\delta_b}

\title[Spectrum of random hyperbolic surfaces of high genus]
{Benjamini-Schramm convergence and spectrum \\ of random hyperbolic surfaces of high genus}

\author{Laura Monk}

\address{Universit\'{e} de Strasbourg, CNRS, IRMA UMR 7501, F-67000 Strasbourg, France}

\date{\today}

\email{monk@math.unistra.fr}



\begin{document}

\maketitle

\begin{abstract}
  We study geometric and spectral properties of \emph{typical} hyperbolic surfaces of high genus, excluding a set of
  small measure for the Weil-Petersson probability measure.
  We first prove Benjamini-Schramm convergence to the hyperbolic plane $\mathcal{H}$ as the genus~$g$ goes to infinity.
  An estimate for the number of eigenvalues in an interval $[a,b]$ in terms of $a$, $b$ and $g$ is then proven using the
  Selberg trace formula. It implies the convergence of spectral measures to the spectral measure of $\mathcal{H}$ as
  $g \rightarrow + \infty$, and a uniform Weyl law as $b \rightarrow + \infty$. We deduce a bound on the number of
  small eigenvalues, and the multiplicity of any eigenvalue.
\end{abstract}

\section{Introduction and main results}

Let $X$ be a compact (oriented, connected, without boundary) hyperbolic surface.  It is isometric to a quotient
$\faktor{\mathcal{H}}{\Gamma}$, where $\mathcal{H} = \{ x + iy, y > 0 \}$ is the upper half-plane equipped with the
metric $\d s^2 = \frac{\d x^2 + \d y^2}{y^2}$, and $\Gamma \subset \mathrm{PSL}_2(\R)$ is a co-compact Fuchsian group.
The metric induces a hyperbolic distance $\dist_{\mathcal{H}}$ on $\mathcal{H}$ and $\dist_X$ on $X$.  The hyperbolic
structure on $X$ induces a volume form $\volhyp_X$ on $X$ and $\volhyp_{\mathcal{H}} = \frac{\d x \d y}{y^2}$ on
$\mathcal{H}$. The total volume of $X$ is $2 \pi(2g-2)$, where $g$ is the genus of $X$. Therefore, studying surfaces of
high genus is equivalent to studying surfaces of large volume.

In the following, the notation $T_1 = \O{T_2}$ means that there exists a universal constant $C>0$ such that
$\abs{T_1} \leq C \; T_2$ for any choice of parameters. If the constant depends on some parameter $x$, then we will write
$T_1 = \mathcal{O}_x(T_2)$.

\subsection{Spectrum of the Laplacian on a hyperbolic surface}

Let $\Delta_X$ be the (positive) Laplace-Beltrami operator on $L^2(X)$, and let $(\lambda_j)_{j \geq 0}$ be its
non-decreasing sequence of eigenvalues (with multiplicities).  For any real numbers $0 \leq a \leq b$, let
$\counting{\Delta}{X}{a,b}$ be the number of eigenvalues of $\Delta_X$ in the interval $[a,b]$.  This article aims at
understanding the behavior of $\counting{\Delta}{X}{a,b}$ as the genus $g$ approaches infinity. To put this paper in
context, here are the known results about these counting functions.
\begin{itemize}
\item $\counting{\Delta}{X}{0,\frac{1}{4}} \leq 2g-2$~\cite{otal2009}, and this bound is
  optimal~\cite{randol1974,buser1977}.
\item The examples from~\cite{buser1977} also prove that,
  for any $\epsilon \in \left( 0, \frac{1}{4} \right)$, there are compact
  hyperbolic surfaces such that $\counting{\Delta}{X}{0,\epsilon} = 2g-2$.
\item On the contrary, for any $\epsilon > 0$, there cannot be a topological bound on
  $\counting{\Delta}{X}{0,\frac{1}{4}+\epsilon}$.
  Indeed, there exist compact hyperbolic surfaces of a given genus
  $g \geq 2$ with an arbitrarily large number of eigenvalues below $\frac{1}{4}+\epsilon$~\cite{buser1977}.
\item The Weyl law gives the asymptotic behavior of $\counting{\Delta}{X}{0,b}$ for a \emph{fixed} surface $X$ as $b$
  goes to infinity. In our setting, the best known estimate is the following~\cite{berard1977,randol1978}:
  \begin{equation*}
    \frac{\counting{\Delta}{X}{0,b}}{\volhyp_X (X)} = \frac{b}{4 \pi} + \mathcal{O}_X\paren*{\frac{\sqrt b}{\log b}}
  \end{equation*}
  where the implied constant depends on the surface $X$.  
\end{itemize}

\subsection{Random compact hyperbolic surfaces of high genus}

The question one may now ask is: what are the spectral properties of a \emph{typical} compact hyperbolic surface? Can we
improve the previous results if we allow ourselves to exclude a \emph{small} set of surfaces?

Our approach to this problem involves working with \emph{random} compact hyperbolic surfaces. There are several ways to
do this. We chose here to work with the Weil-Petersson volume, following Mirzakhani's approach~\cite{mirzakhani2013},
but there is also, for instance, a construction of Brooks and Makover~\cite{brooks2004}, in which similar work can
probably be done.

In the following, our probability space will be the \emph{moduli space} $\mathcal{M}_g$ of compact hyperbolic surfaces
of genus $g$. It is the set of all compact hyperbolic surfaces of genus $g$, up to isometry. The moduli space is equipped
with a natural symplectic form $\omegawp$ called the \emph{Weil-Petersson form}~\cite{weil1958}. This
induces a volume form $\volwp = \frac{(\omegawp)^{\wedge(3g-3)}}{(3g-3)!}$ on $\mathcal{M}_g$, which is of finite volume
$V_g = \volwp(\mathcal{M}_g)$. Therefore, it can be renormalized to obtain a probability measure
$\Pwp = \frac{1}{V_g} \; \volwp$.

We say that an event $\mathcal{A}_g \subset \mathcal{M}_g$ occurs \emph{with high probability} when
$\Pwp(\mathcal{A}_g) \rightarrow 1$ as $g \rightarrow + \infty$. We will prove results true in that sense, and hence not
for all surfaces but most.

This probabilistic approach has proven to be a good method in the study of graphs~\cite{erdos1960}. A significant
example of a spectral result true with high probability is Friedman's theorem for random large regular
graphs~\cite{friedman2008}, first conjectured by Alon~\cite{alon1986}.  Random regular graphs and random compact
hyperbolic surfaces share many geometric properties. In this article, we explore the idea that this resemblance between
graphs and surfaces is not only geometric but also spectral; an idea motivated by the deep connection between geometry
and spectrum, in both settings.

\subsection{Geometry of random surfaces}

Multiple aspects of the geometry of random surfaces of high genus have already been studied. In~\cite{mirzakhani2013},
Mirzakhani estimated various geometric quantities (the injectivity radius, Cheeger constant, diameter...) for typical
random surfaces.

\subsubsection*{Radius of injectivity}
The \emph{radius of injectivity} $\injrad_z(X)$ at a point $z$ of a compact hyperbolic surface $X$  is the supremum of
all real numbers $r \geq 0$ such that the ball of radius $r$ centered at $z$ in $X$ is isometric to a ball in the
hyperbolic plane $\mathcal{H}$.  The \emph{global radius of injectivity} $\injrad (X)$ is the infimum over $X$ of the
radius of injectivity, or equivalently twice the length of the shortest closed geodesic on $X$.

In the following, we will work under the assumption that the random surfaces we are considering are a \emph{uniformly
  discrete} family, that is to say that their radius of injectivity is bounded below by a constant $r_g$. In order to guarantee
this, we will have to exclude some surfaces. We will control the probability measure of the set of excluded surfaces
using the following result.
\begin{theo}[Theorem 4.2 in \cite{mirzakhani2013}]
  \label{theo:injrad}
  There exists a constant $C >0$ such that, for any large enough $g$ and any
  small enough $r>0$,
  \begin{equation}
    \label{eq:injrad}
    \frac{1}{C} \; r^2 \leq \Pwp \paren*{\injrad (X) \leq r} \leq C \; r^2.
  \end{equation}
\end{theo}
The lower bound of this statement motivates the fact that our uniform discreteness bound $r_g$ will need to go to zero
as $g$ approaches infinity in order for the event to occur with high probability.

\subsubsection*{Benjamini-Schramm convergence} The notion of Benjamini-Schramm convergence has first been introduced by
Benjamini and Schramm in the context of sequences of graphs~\cite{benjamini2001}, but can naturally be
extended to a continuous setting (see~\cite{abert2011,abert2017,bowen2015}). There is a general
definition of Benjamini-Schramm convergence for a \emph{deterministic} sequence of hyperbolic surfaces $(X_g)_g$. In the
special case when the limit of $(X_g)_g$ is the hyperbolic plane $\mathcal{H}$, it is equivalent to the following
simpler property:
\begin{equation}
  \label{eq:bs_def}
  \forall L>0, \lim_{g \rightarrow + \infty} \frac{\volhyp_{X_g} \paren*{\acc{z \in X_g \; : \; \injrad_{z}(X_g) < L}}}{\volhyp_{X_g} (X_g)} = 0
\end{equation}
which we will therefore use as a definition.  The idea behind this characterization is the following. We consider a
distance $L>0$. On the (fixed) surface $X_g$, one can pick a point~$z$ at random, using the normalized measure
$\frac{1}{\volhyp_{X_g}(X_g)} \; \volhyp_{X_g}$. Equation~\eqref{eq:bs_def} means that the probability for the ball of
center $z$ and radius $L$ to be isometric to a ball in the hyperbolic plane goes to one as $g \rightarrow
+ \infty$.

The case where $\injrad (X_g) \rightarrow + \infty$ as $g \rightarrow + \infty$ is a simple situation which implies
Benjamini-Schramm convergence towards $\mathcal{H}$. However, Theorem~\ref{theo:injrad} proves that this does not occur with high
probability for our probabilistic model.

Our first result is a quantitative estimate of the Benjamini-Schramm speed of convergence of random hyperbolic surfaces
of high genus towards $\mathcal{H}$.
\begin{theo}
  \label{theo:bs}
  For any $g \geq 2$ and any $L, M >0$, there exists a set $\mathcal{A}_{g,L,M} \subset \mathcal{M}_g$ such that
  for any hyperbolic surface $X \in \mathcal{A}_{g,L,M}$, 
  \begin{equation}
    \label{eq:bs}
    \volhyp_X \paren*{\acc*{ z \in X \; : \; \injrad_z (X) < L}} \leq e^{L} M,
  \end{equation}
  and $1 - \Pwp(\mathcal{A}_{g,L,M}) = \O{\dfrac{L e^{2L}}{M}}$.
\end{theo}
We recall that the implied constant is independent of the genus $g$ and the parameters $L$, $M$.
Since the total area of a compact hyperbolic surface of genus $g$ is $2 \pi (2g-2)$, this result will only be
interesting for $L \leq \log g$.
Specifying the parameters to be $L = \frac{1}{6} \; \log g$, $M =
g^{\frac{1}{2}}$ and $r_g =  g^{- \frac{1}{24}} (\log g)^{\frac{9}{16}}$,
Theorem~\ref{theo:injrad} and \ref{theo:bs} together lead to the following corollary.
\begin{coro}[Geometric assumptions]
  \label{coro:geometry}
  For large enough $g$, there exists a subset $\mathcal{A}_g \subset \mathcal{M}_g$ such that, for any
  hyperbolic surface $X \in \mathcal{A}_g$,
  \begin{align}
    & \injrad(X) \geq  g^{- \frac{1}{24}} (\log g)^{\frac{9}{16}} \\
    \label{eq:bs_applied}
    & \dfrac{\volhyp_X \paren*{\{ z \in X \; : \; \injrad_z(X) < \frac 1 6 \log g \}}}{\volhyp_X (X)}  = \O{ g^{-\frac{1}{3}}}
  \end{align}
  and $1 - \Pwp(\mathcal{A}_g) = \O{g^{-\frac{1}{12}} (\log g)^{\frac 9 8}}$.
\end{coro}
The estimate~\eqref{eq:bs_applied} is similar to the one proved in~\cite[Section 4.4]{mirzakhani2013}. The
proof is the same, though an incorrect argument has been modified here. 

The rest of this article is devoted to the study of the spectral properties of a given element $X \in \mathcal{A}_g$,
where $\mathcal{A}_g$ is the set from Corollary~\ref{coro:geometry}.  There are no more probabilistic arguments from
this point, and the spectral results we prove in the following could be adapted to any setting in which a property
similar to Corollary \ref{coro:geometry} holds. Note that the final results depend on Corollary \ref{coro:geometry} in
a way made explicit by Lemmas \ref{lemm:kernel_sum_B} and \ref{lemm:kernel_sum_H}.

\subsection{Spectrum of random surfaces}

Geometry and spectrum of Riemannian manifolds, and especially of hyperbolic surfaces, are connected in a variety of
ways.  A straightforward example of this interaction can be found in Cheeger and Buser's inequalities
\begin{equation*}
\frac{h^2}{4} \leq \lambda_1 \leq h(1+10h),
\end{equation*}
where $h$ is the Cheeger constant of the hyperbolic surface $X$~\cite{cheeger1970,buser1982}.
Mirzakhani deduced from her probabilistic uniform lower bound on the Cheeger constant that there exists a constant $c_0>0$
such that, with high probability, $\lambda_1 \geq c_0$~\cite{mirzakhani2013}.
This article aims at providing more information on the distribution of eigenvalues for most random surfaces.

\subsubsection*{Counting functions}

Our main results are the following two estimates on the counting functions $\counting{\Delta}{X}{a,b}$, true with
high probability.

\begin{theo}
  \label{theo:upper_bound}
  For any large enough $g$, any $0 \leq a \leq b$ and any $X \in \mathcal{A}_g$ from Corollary \ref{coro:geometry}, 
  \begin{equation}
    \label{eq:upper_bound}
    \frac{\counting{\Delta}{X}{a,b}}{\volhyp_X (X)} = \O{b-a + \sqrt{\frac{b+1}{\log g}}}.
  \end{equation}
    Whenever $b \leq \frac 1 4$, if we set $c = 2^{-15}$, we also have 
    \begin{equation}
      \label{eq:upper_bound_small_eigenvalue}
    \frac{\counting{\Delta}{X}{0,b}}{\volhyp_X(X)} 
    = \O{\frac{g^{-c \paren*{\frac{1}{4}-b}^2}}{\paren*{\log g}^{\frac 3 4}}}.
  \end{equation}
\end{theo}
This theorem provides us with an upper bound on the number of eigenvalues in an interval $[a,b]$.  In
equation~\eqref{eq:upper_bound}, the term $\sqrt{\frac{b+1}{\log g}}$ corresponds to a minimum scale, below which we can
have no additional information by shrinking the spectral window.

The second part of the statement controls the number of small eigenvalues. When we take $b = \frac 1 4$, we
obtain that $\counting{\Delta}{X}{0,\frac 1 4}$ is $\O{g \; (\log g)^{-\frac 3 4}}$. This is better than the
(optimal) deterministic estimate $\counting{\Delta}{X}{0, \frac 1 4} \leq 2g-2$. Furthermore, this bound becomes
better the further $b$ is from the bulk spectrum $\left[ \frac 1 4, + \infty \right)$. As a consequence, the examples of
surfaces with $2g-2$ eigenvalues in $[0, \epsilon]$ (for arbitrarily small $\epsilon$) are not typical.

The second result is a more precise approximation of the counting functions. 
\begin{theo}
  \label{theo:equivalent}
  There exists a universal constant $C>0$ such that, for any large enough $g$, any $0 \leq a \leq b$ and any hyperbolic
  surface $X \in \mathcal{A}_g$ from Corollary \ref{coro:geometry}, one can write the counting function
  $\counting{\Delta}{X}{a,b}$ as
  \begin{equation*}
    \frac{\counting{\Delta}{X}{a,b}}{\volhyp_X (X)}
    = \integralterm + R(X,a,b)
  \end{equation*}
  where 
  \begin{equation*}
    - C \sqrt{\frac{b+1}{\log g}} \leq R(X,a,b) \leq  C \sqrt{\frac{b+1}{\log g}} \; \log \paren*{2 + (b-a)
      \sqrt{\frac{\log g}{b+1}}}^{\frac 1 2}.
  \end{equation*}
\end{theo}

We recall that the sets $\mathcal{A}_g$ from Corollary~\ref{coro:geometry} are fixed subsets of $\mathcal{M}_g$, of
probability going to $1$ as $g \rightarrow + \infty$. Therefore, the previous results hold \emph{with high probability},
uniformly with respect to the parameters $a$ and $b$.
The remainder $R(X,a,b)$ is negligible compared to the main term as soon as the size $b-a$ of the spectral window
is much larger than the minimal spacing $\sqrt{\frac{b+1}{\log g}}$. 
There are several limits that can be interesting to study.
\begin{itemize}
\item When $[a,b]$ is fixed and $g \rightarrow + \infty$, our result shows that the spectrum of $\Delta_X$ approaches
  the continuous spectrum of the Laplacian on $\mathcal{H}$. This is the analogous of the fact that, if a sequence of
  $d$-regular graphs converges in the sense of Benjamini-Schramm to the $d$-regular tree, then their spectral measure
  converges to the Kesten-McKay law~\cite{anantharaman2017}.
\item By taking $a = 0$ and $b$ going to infinity, one can recover a uniform Weyl law, with remainder of order
  $\mathcal{O}_g \paren*{\sqrt{b \log b}}$. The constant is independent of the surface, and explicit in terms of
  $g$. We could probably have obtained a better remainder (for instance, $\mathcal{O}_g \paren*{\frac{\sqrt b}{\log b}}$
  as in~\cite{berard1977,randol1978}) had we allowed the probability set $\mathcal{A}_g$ to depend on the
  parameter $b$, which we did not do here in order to make the discussion simpler.
\item One can also consider mixed regimes, where both $b$ and $g$ go to infinity. 
\end{itemize}

Our approach in proving the two theorems is inspired by~\cite[Part 9]{lemasson2017}, where Le Masson
and Sahlsten prove the convergence of $\frac{\counting{\Delta}{X_g}{a,b}}{\volhyp_{X_g}(X_g)}$ to $\integralterm$
as $g \rightarrow + \infty$, for a uniformly discrete sequence of compact hyperbolic surfaces $(X_g)_g$ converging
to $\mathcal{H}$ in the sense of Benjamini-Schramm. Here, we do not consider a sequence but a fixed surface of high
genus. Furthermore, we estimate precisely the error term, which leads us to considering different kernels in the
trace formula.

\subsubsection*{Eigenvalue multiplicity and $j$-th eigenvalue}

For a compact hyperbolic surface $X \in \mathcal{M}_g$, and a real number $\lambda > 0$, let $m_X(\lambda)$ denote the
multiplicity of the eigenvalue $\lambda$ of $\Delta_X$.  We can estimate $m_X(\lambda)$ with high probability, using
Theorem \ref{theo:upper_bound} and a shrinking spectral window around the eigenvalue $\lambda$.
\begin{coro}
  \label{coro:multiplicity}
  There exists a universal constant $C>0$ such that, for any large enough $g$, any $\lambda \geq 0$ and any hyperbolic
  surface $X \in \mathcal{A}_g$ from Corollary \ref{coro:geometry},
  \begin{equation*}
    \frac{m_X(\lambda)}{g} \leq C \sqrt{\frac{1+\lambda}{\log g}} \cdot
  \end{equation*}
  If furthermore $\lambda \leq \frac 1 4 - \epsilon$, if we set $c = 2^{-15}$,
  \begin{equation*}
    \frac{m_X(\lambda)}{g} \leq C \frac{g^{- c \epsilon^2}}{(\log g)^{\frac 3 4}} \cdot
  \end{equation*}
\end{coro}

Another probabilistic upper bound $\frac{m_X(\lambda)}{g} = \mathcal{O}_\lambda \paren*{\frac{1}{\log g}}$ and
$\frac{m_X(\lambda)}{g} = \O{g^{-c' \sqrt{\epsilon}}}$ when $\lambda \leq \frac 1 4 - \epsilon$ has been proved recently
in~\cite{gilmore2019}. This was achieved by estimating the $L^p$-norms of eigenfunctions on random hyperbolic surfaces
of high genus. Though the behavior in terms of $g$ of the bound $\frac{1}{\log g}$ is better than our
$\frac{1}{\sqrt{\log g}}$, the implied constant depends on $\lambda$.

One can also deduce from Theorem \ref{theo:equivalent} an estimate on the $j$-th eigenvalue $\lambda_j(X)$ of $X$, in
terms of $j$ and $g$, true with high probability.

\begin{coro}
  \label{coro:jth}
  There exists a universal constant $C>0$ such that, for any large enough $g$, any $j \geq 0$ and any hyperbolic surface $X \in
  \mathcal{A}_g$ from Corollary \ref{coro:geometry},
  \begin{equation*}
    \abs*{\lambda_j(X) - \frac j g} \leq C \paren*{1 + \sqrt{\frac{j}{g} \log \paren*{2+\frac j g}}}.
  \end{equation*}
\end{coro}
There are two interesting regimes in which one can apply this corollary:
\begin{itemize}
\item If $j \leq A g$ for a $A \geq 1$, then $\lambda_j(X) = \O{A}$.
\item If $j \gg g$, then $\lambda_j(X) \sim \frac j g$ uniformly in $X$.
\end{itemize}
As a consequence, the multiplicity of the $j$-th eigenvalue $\lambda_j(X)$ of a typical compact hyperbolic surface
$X \in \mathcal{A}_g$ satisfies
\begin{equation}
\frac{m_X(\lambda_j(X))}{g} = \O{\sqrt{\frac{1+ \frac j g}{\log g}}},
\end{equation}
which is an improvement of the deterministic estimate $m_X(\lambda_j(X))\leq 4g+2j+1$ from~\cite{besson1980}.

\subsubsection*{Perspectives}

Inspired by the similar geometric and spectral properties of large regular graphs and large genus hyperbolic surfaces,
Wright conjectured~\cite{wright2020} that Friedman's theorem holds for random hyperbolic surfaces taken with the
Weil-Petersson probability measure.

\begin{conj}
  For any small enough $\epsilon > 0$, 
  \begin{equation*}
    \lim_{g \rightarrow + \infty} \Pwp \paren*{\counting{\Delta}{X}{0, \frac 1 4 - \epsilon} = 1} = 1.
  \end{equation*}
  In other words, for any fixed $\epsilon > 0$, with high probability, there is no non-trivial eigenvalue smaller than
  $\frac 1 4 - \epsilon$.
\end{conj}

It might be possible to prove this result using the Selberg trace formula with an adequate test function. However,
obtaining an estimate as precise as $\counting{\Delta}{X}{0, \frac 1 4 - \epsilon} = 1$ would be highly technical, and
cannot be achieved using the approach presented in this article. 

\subsection*{Organization of the paper}

The paper is organized as follows.

We start by proving Benjamini-Schramm convergence for the Weil-Petersson probabilistic model in Section~\ref{sec:bs}.
The proof of Theorems~\ref{theo:upper_bound} and \ref{theo:equivalent} then spans over Section~\ref{sec:B} and
\ref{sec:H}, which corresponds respectively to the case where $0 \leq a \leq b \leq 1$ (bottom of the spectrum) and
$\frac 1 2 \leq a \leq b$ (away from small eigenvalues). The way the different parts depend on one another is explained
in Figure \ref{fig:dependencies} for clarity.
We finish by proving Corollary~\ref{coro:jth}  in Section~\ref{sec:coro}.

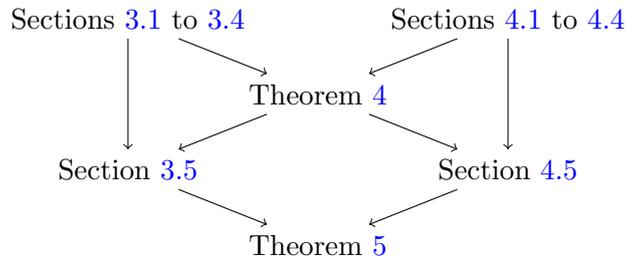
\begin{figure}[h]
  \centering
  \begin{tikzpicture}
    \node (SA1) at (-2.5,1) {Sections~\ref{sec:test_function_B} to \ref{sec:upper_bound_B}};
    \node (SB1) at (2.5,1) {Sections~\ref{sec:test_function_H} to \ref{sec:upper_bound_H}};
    \node (T1) at (0,0) {Theorem~\ref{theo:upper_bound}};
    \node (SA2) at (-2.5,-1) {Section~\ref{sec:comparison_B}};
    \node (SB2) at (2.5,-1) {Section~\ref{sec:comparison_H}};
    \node (T2) at (0,-2) {Theorem~\ref{theo:equivalent}};

    \draw[->] (SA1) -- (T1);
    \draw[->] (SB1) -- (T1);

    \draw[->] (SA1) -- (SA2);
    \draw[->] (SB1) -- (SB2);

    \draw[->] (T1) -- (SA2);
    \draw[->] (T1) -- (SB2);

    \draw[->] (SA2) -- (T2);
    \draw[->] (SB2) -- (T2);
  \end{tikzpicture}
  \caption{The steps of the proofs of Theorems~\ref{theo:upper_bound} and \ref{theo:equivalent}, and the way they depend
    on one another. The left part corresponds to the case $b \leq 1$, and the right part to the case $a \geq \frac 1 2$.}
  \label{fig:dependencies}
\end{figure}

\subsection*{Acknowledgements}

The author would like to thank Nalini Anantharaman, Alix Deleporte and Etienne Le Masson for valuable discussions and comments.

\section{Benjamini-Schramm convergence of random surfaces}
\label{sec:bs}

We here proceed to the proof of Theorem~\ref{theo:bs}, which states that, with high probability, random surfaces are
close to the hyperbolic plane in the Benjamini-Schramm sense. It is a generalization of a result proved in~\cite[Section
4.4]{mirzakhani2013}.
 
\begin{proof}[Proof of Theorem~\ref{theo:bs}]
  Let $X$ be a compact hyperbolic surface of genus $g$. In order to estimate the volume of
  \begin{equation*}
    X^-(L) = \acc*{z \in X \; : \; \injrad_z(X) < L},
  \end{equation*}
  we establish a link between this volume and the number of small geodesics on $X$.  
  
  Let $z$ be a point in $X$ of radius of injectivity $r < L$. There is a simple geodesic arc $c$ in $X$ of length $2r$ based
  at $z$, which is freely homotopic to a closed geodesic $\gamma$ of length $\ell \leq 2r$. Let us bound the distance
  between $z$ and $\gamma$; this way, we will be able to say $z$ belongs in a neighborhood of $\gamma$ of small volume.

  \begin{figure}[h]
    \begin{center}
      \subfloat[On the surface $X$]{
        \includegraphics[scale=0.65]{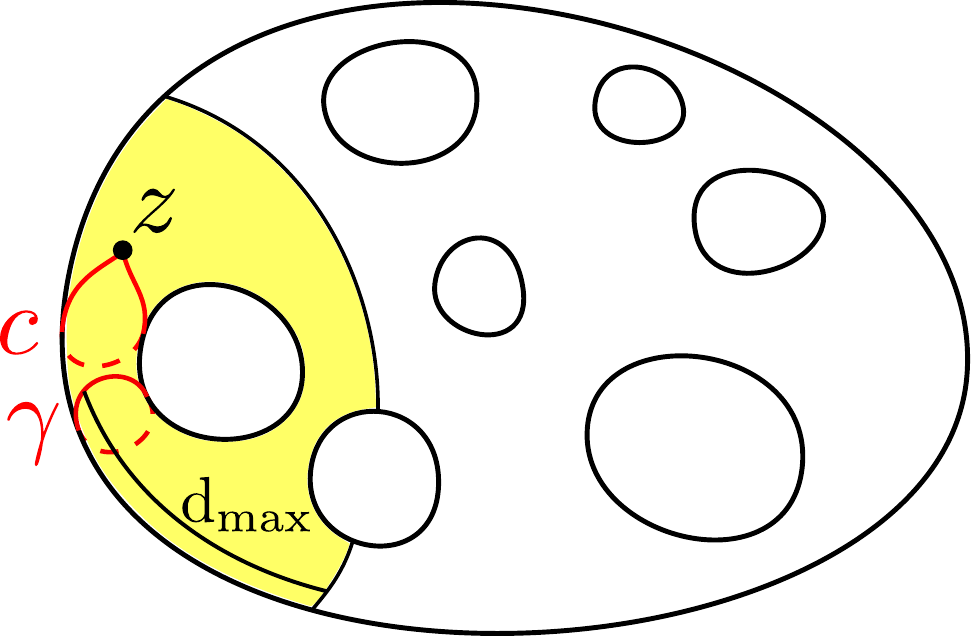}
        \label{sub:X}
      }
      \subfloat[On the universal cover $\mathcal{H}$]{
        \includegraphics[scale=0.55]{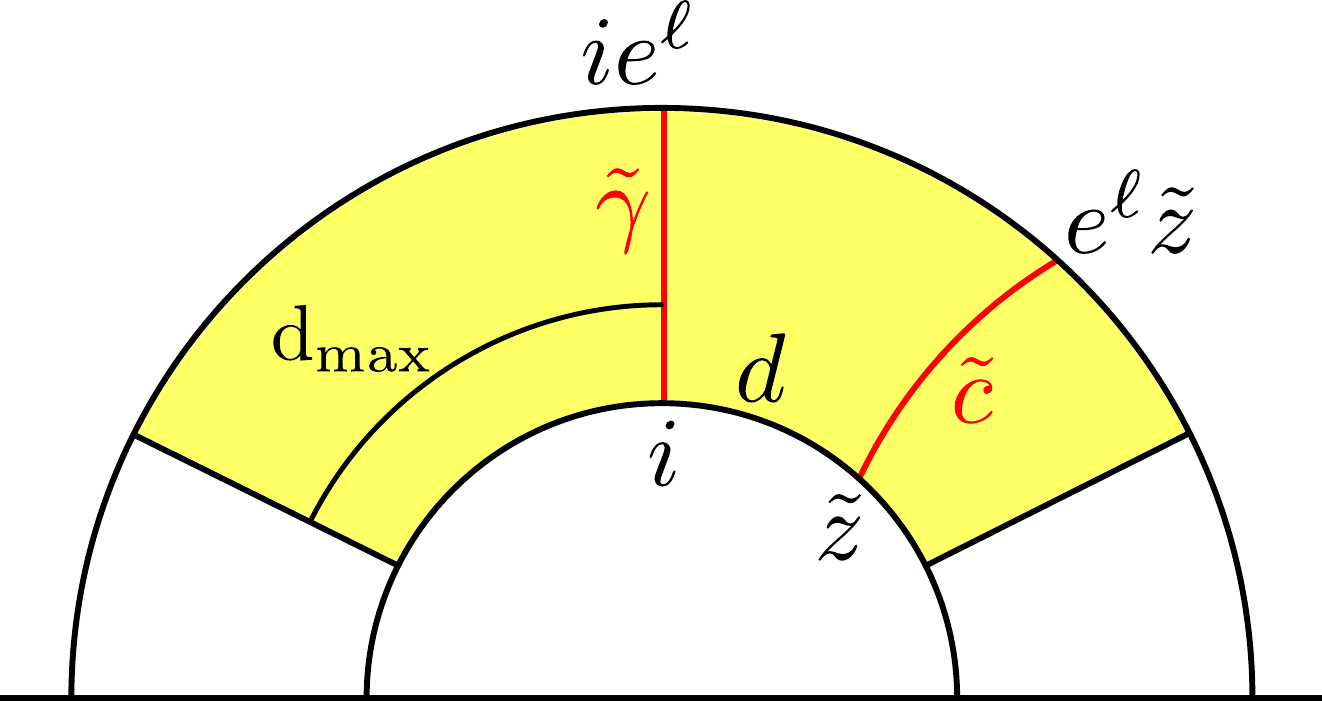}
        \label{sub:H}
      }
      \caption{Illustration of the geometric construction in the proof of Theorem~\ref{theo:bs}.}
      \label{fig:geometry_lemma}
    \end{center}
  \end{figure}

  By lifting $z$, $c$ and $\gamma$ to the hyperbolic plane and applying an isometry, we reduce the problem to the
  situation represented in Figure \ref{fig:geometry_lemma}: the geodesic $\gamma$ is lifted to the geodesic segment
  between $i$ and $e^\ell i$, and $c$ to the segment between a point $\tilde{z} = x + iy$ of modulus~$1$, and
  $e^\ell \tilde{z}$. Let us bound the distance $d$ between $\tilde{z}$ and $\tilde{\gamma}$. By usual expressions for
  the hyperbolic distance in the Poincar\'e half-plane model (see~\cite[Theorem 1.2.6]{katok1992} for instance),
  \begin{equation*}
    \cosh (d)
    = \cosh (\dist_{\mathcal{H}}(\tilde{z}, i))
    = 1 + \frac{|\tilde{z}-i|^2}{2y}
    = 1 + \frac{x^2 + (y-1)^2}{2y}
    = \frac 1 y 
  \end{equation*}
  and
    \begin{equation*}
    \sinh (r) = \sinh \left( \frac{\dist_{\mathcal{H}}(\tilde{z},e^\ell \tilde{z})}{2} \right)
    = \frac{1}{2} \frac{|\tilde{z}e^\ell - \tilde{z}|}{\sqrt{y^2 e^\ell}}
    = \frac{\sinh \paren*{\frac \ell 2}}{y} \cdot
  \end{equation*}
  As a consequence,
  \begin{equation*}
    \cosh (d) = \frac{\sinh (r)}{\sinh \paren*{\frac \ell 2}}
    \leq \frac{e^{L}}{2 \sinh \paren*{\frac{\ell}{2}}} =:
    \cosh (\dist_{\mathrm{max}}(\gamma,L)). 
  \end{equation*}
  Then $z$ belongs to the $\dist_{\mathrm{max}}(\gamma,L)$-neighborhood of the closed geodesic $\gamma$ in $X$.

  The volume of this neighborhood is less than the volume of the corresponding collar in the cylinder of central
  geodesic $\gamma$, which can be computed using the Fermi coordinates:
  \begin{equation*}
    \int_0^\ell \int_{-\dist_{\mathrm{max}}(\gamma,L)}^{\dist_{\mathrm{max}}(\gamma,L)} \cosh (\rho) \d \rho \d t = 2 \ell \sinh (\dist_{\mathrm{max}}(\gamma,L))
    \leq 2 \ell \cosh(\dist_{\mathrm{max}}(\gamma,L)) = \frac{\ell \; e^{L}}{\sinh \paren*{\frac{\ell}{2}}} \cdot
  \end{equation*}
  Since $x \leq \sinh x$, the volume of the $\dist_{\mathrm{max}}(\gamma,L)$-neighborhood of $\gamma$ in $X$ is
  smaller than~$2 e^{L}$.  As a consequence, any point $z \in X^-(L)$ is in a neighborhood of volume less than $2 e^{L}$
  around a simple closed geodesic of length at most $2 r \leq 2 L$. This implies:
    \begin{equation*}
    \volhyp_X (X^-(L)) \leq 2  e^{L} \; \counting{\ell}{X}{2 L},
  \end{equation*}
  where, for any positive $L$, $\counting{\ell}{X}{L}$ is the number of simple closed geodesics with length at most $L$
  on the hyperbolic surface $X$.
  Then, on the set
  \begin{equation*}
    \mathcal{A}_{g,L,M} = \left\{ X \in \mathcal{M}_g \; : \; \counting{\ell}{X}{2 L} \leq \frac{M}{2} \right\}
  \end{equation*}
  equation \eqref{eq:bs} is proved.

  Let us estimate the Weil-Petersson probability of this event using Markov's inequality:
  \begin{equation*}
    \Pwp(\mathcal{M}_g \smallsetminus \mathcal{A}_{g,L,M})
     = \frac{1}{V_g} \int_{\mathcal{M}_g} \1{ \left\{ \counting{\ell}{X}{2 L} > \frac{M}{2} \right\} } \d \volwp(X) 
     \leq \frac{2}{M} \; \Ewp [\counting{\ell}{X}{2 L}]
   \end{equation*}
  where, for a function $F$ on the moduli space, the Weil-Petersson expectation is defined by:
  \begin{equation*}
    \Ewp [F] = \frac{1}{V_g} \int_{\mathcal{M}_g} F(X) \d \volwp(X).
  \end{equation*}
  The expectation of the length counting function $\counting{\ell}{X}{L}$ is estimated in~\cite[Lemma
  4.1]{mirzakhani2013}. We sum up the equations (4.1) with $k=1$ and (4.3) for $k$ between $1$ and
  $\lfloor \frac{g}{2} \rfloor$ in this lemma, with the length $2 L$, and use equation (3.19), to prove that, as
  $g \rightarrow + \infty$:
  \begin{equation*}
    \Ewp [ \counting{\ell}{X}{2 L} ] = \O{L e^{2L}}.
  \end{equation*}
\end{proof}

\section{Proof of Theorems~\ref{theo:upper_bound} and \ref{theo:equivalent} at the bottom of the spectrum}
\label{sec:B}

Let us prove Theorems~\ref{theo:upper_bound} and \ref{theo:equivalent} in the case where the spectral window is bounded
above, more precisely $0 \leq a \leq b \leq 1$.  The reason why we make the assumption $b \leq 1$ is that our choice
of test function behaves poorly for large values of $b$. The value $1$ is arbitrary, and could be replaced by any fixed
value larger than $\frac 1 2$ (because the proof in the case away from small eigenvalues only works for
$a \geq \frac 1 2$).

\subsection{Trace formula, test function and sketch of the proof}
\label{sec:test_function_B}

\subsubsection*{The Selberg trace formula}

Our main tool in this proof is the Selberg trace formula, which can be expressed the following way.

\begin{theo}[Selberg trace formula~\cite{selberg1956}]
  \label{theo:selberg}
  Let $X = \faktor{\mathcal{H}}{\Gamma}$ be a compact hyperbolic surface.  Let $(\psi_j)_{j \geq 0}$ be an orthonormal
  basis of eigenfunctions of the Laplacian in $L^2(X)$ of associated non decreasing sequence of eigenvalues
  $(\lambda_j)_{j \geq 0}$.  Let $h : \C \rightarrow \C$ be a function satisfying:
  \begin{enumerate}
  \item $h(-r) = h(r)$ for any $r \in \C$;
  \item $h$ is analytic in the strip $\abs{\Im z} \leq \frac{1}{2} + \epsilon$ for some $\epsilon > 0$;
  \item for all $r$ in that strip,
    \begin{equation}
      \label{eq:trace_decrease}
      h(r) = \O{ \frac{1}{(1+\abs{r}^2)^{1+\epsilon}}}.
    \end{equation}
  \end{enumerate}
  Then the following formula holds (with every term well-defined and converging):
  \begin{equation}
    \label{eq:selberg}
    \sum_{j=0}^{+ \infty} f(\lambda_j) = \frac{\volhyp_X(X)}{4 \pi} \int_{\frac{1}{4}}^{+ \infty} f(\lambda) \tanh \paren*{\pi
      \sqrt{\lambda - \frac{1}{4}}} \d \lambda + \int_D
    \sum_{\gamma \in \Gamma \setminus \{ \mathrm{id} \}} K(z, \gamma \cdot z) \d \volhyp_{\mathcal{H}} (z)
  \end{equation}
  where $f : [0,+\infty) \rightarrow \C$ is defined by $f \left( \frac{1}{4} + r^2 \right) = h(r)$, $D$ is a fundamental
  domain of the action of $\Gamma$ on $\mathcal{H}$
  and $K(z,w) = K(\dist_{\mathcal{H}}(z,w))$ is the kernel associated to $h$. $K$ has the following expression:
  \begin{equation}
    \label{eq:selberg_transform}
    K(\rho) = - \frac{1}{\sqrt 2 \; \pi} \int_{\rho}^{+ \infty} \frac{g'(u)}{\sqrt{\cosh u - \cosh \rho}} \; \d u,
  \end{equation}
  where $g(u) = \frac{1}{2 \pi} \int_{- \infty}^{+ \infty} h(r) \; e^{iru} \d r$ is the inverse Fourier transform of the function $h$.
\end{theo}

Here is a brief description of the terms in equation~\eqref{eq:selberg}.
\begin{itemize}
\item The left hand term is a spectral average with density $f$. For instance, if $f$ resembles the indicator
  function of a segment (as will happen in the following), it will be close to the number of eigenvalues in that
  spectral window.
\item The first term on the right hand side does not depend on the metric on $X$ but only on its topology, because the
  volume of any hyperbolic surface of genus $g$ is $2 \pi(2g-2)$. If $f$ approaches a step function, then it will converge to
  the integral term in Theorem \ref{theo:equivalent}.
\item The last term is a geometric term; it is often expressed in terms of the lengths of closed geodesics on $X$
  (see~\cite[Theorem 9.5.3]{buser1992}), but this integral expression will be more convenient for our purposes. This is
  the term for which the geometric assumptions from Corollary~\ref{coro:geometry} will be needed. Its size will depend
  on the regularity of the test function, because of the presence of the Fourier transform. This phenomenon will limit
  the speed at which the test function we use can approximate the indicator function of $[a,b]$.
\end{itemize}

\subsubsection*{The test function}

Let $0 \leq a \leq b \leq 1$. The test function we are going to use in this section is defined by
\begin{equation*}
  h_t(r) = f_t \paren*{\frac 1 4 + r^2},
\end{equation*}
where, for $\lambda \geq 0$,
\begin{equation*}
  f_t(\lambda) = (\1{[a,b]} \star v_t)\paren*{\lambda} = \frac{t}{\sqrt \pi} \int_a^b \exp \paren*{-t^2(\lambda-\mu)^2} \d \mu. 
\end{equation*}
Here, $t>0$ is a parameter that will grow like $\sqrt{\log g}$, and $v_t(x) = \frac{t}{\sqrt \pi} \; \exp \paren*{-t^2 \rho^2}$ is
the centered normalized Gaussian of variance $\frac 1 t$. As a consequence, $f_t$ is a smooth (pointwise) approximation of the
function $\unreg{[a,b]}$ defined by
\begin{equation*}
  \unreg{[a,b]}(x) = \left\{
    \begin{array}{ll}
      1& \text{if } x \in (a,b) \\
      \frac 1 2 & \text{if } x = a \text{ or } x=b \\
      0 & \text{otherwise.}
    \end{array}
\right.
\end{equation*}

It is clear that $h_t : \C \rightarrow \C$ is analytic and even. In order to apply the trace formula, we need to
prove an estimate like the one in equation~\eqref{eq:trace_decrease}, which is the aim of the following lemma.
\begin{lemm}
  Let $0 \leq a \leq b$ and $t > 0$.
  There exists a constant $C = C(a,b,t) >0$ such that
  \begin{equation*}
    \forall x \in \R, \forall y \in [-1,1], \abs{h_t(x+iy)} \leq C \exp (-t^2 x^4 + 2 t^2(3+b) x^2). 
  \end{equation*}
\end{lemm}
\begin{proof}
  Let us write
  \begin{equation*}
    h_t(r) = \frac{t}{\sqrt \pi} \int_{a-\frac 1 4}^{b - \frac 1 4} \exp \paren*{-t^2 \paren*{r^2 - \mu}^2} \d \mu.
  \end{equation*}
  The modulus of the integrand for a $r = x + iy$, $-1 \leq y \leq 1$ is
  \begin{align*}
    \abs*{\exp \paren*{-t^2 \paren*{r^2 - \mu}^2}}
    & = \exp \paren*{-t^2 \paren*{ x^4 + y^4 + \mu^2 - 6x^2y^2 - 2 \mu x^2 + 2 \mu y^2}} \\
    & \leq \exp \paren*{-t^2 x^4 + 6 t^2 x^2 + 2 t^2 \paren*{b-\frac 1 4} x^2+ 2t^2 \abs*{\frac 1 4 - a}}.
  \end{align*}
  We then integrate this inequality between $a-\frac 1 4$ and $b - \frac 1 4$.
\end{proof}
This exponential decay guarantees a polynomial decay in the strip $\{ \abs*{\Im r} \leq 1 \}$ like the one required in
the trace formula. Therefore, one can apply it to $h_t$, and rewrite equation~\eqref{eq:selberg} as:
\begin{equation}
  \label{eq:trace_applied_B}
    \frac{1}{\volhyp_X (X)} \sum_{j=0}^{+ \infty} f_t(\lambda_j)
  = \frac{1}{4 \pi} \int_{\frac{1}{4}}^{+\infty} \1{[a,b]}(\lambda)  \tanh \paren*{\pi \sqrt{\lambda - \frac{1}{4}}}
  \d \lambda
  + \mathcal{R}_I(t,a,b) + \mathcal{R}_K(X,t,a,b)
\end{equation}
where
\begin{align}
  \mathcal{R}_I(t,a,b)
  & = \frac{1}{4 \pi} \int_{\frac 1 4}^{+ \infty} (f_t(\lambda) - \1{[a,b]}(\lambda))
    \tanh \paren*{\pi \sqrt{\lambda - \frac{1}{4}}}  \d \lambda, \\
  \mathcal{R}_K(X,t,a,b)
  & = \frac{1}{\volhyp_X (X)} \int_D \sum_{\gamma \in \Gamma \setminus \{\id\}} K_t(z, \gamma \cdot z) \dvolhyp_{\mathcal{H}} (z),
\end{align}
$K_t$ is the kernel associated to $h_t$ and $D$ is a fundamental domain of $X = \faktor{\mathcal{H}}{\Gamma}$.

\subsubsection*{Sketch of the proof}

There are four steps in the proof of Theorems~\ref{theo:upper_bound} and \ref{theo:equivalent} at the bottom of the spectrum.
\begin{itemize}
\item Control the integral term $\mathcal{R}_I(t,a,b)$ (Section~\ref{sec:integral_term_B}).
\item Estimate the geometric term $\mathcal{R}_K(X,t,a,b)$ using the uniform discreteness and the Benjamini-Schramm
  convergence assumptions (Section~\ref{sec:geometric_term_B}).
\item Control $\counting{\Delta}{X}{a,b}$ with $\sum_{j=0}^{+ \infty} f_t(\lambda_j)$, and deduce
  Theorem~\ref{theo:upper_bound}  (Section~\ref{sec:upper_bound_B}).
\item Compare more precisely the sum $\sum_{j=0}^{+ \infty} f_t(\lambda_j)$ to $\counting{\Delta}{X}{a,b}$, and conclude
  (Section~\ref{sec:comparison_B}).
\end{itemize}

\subsection{The integral term}
\label{sec:integral_term_B}

In order to control $\mathcal{R}_I(t,a,b)$, we need to know more about the speed of convergence of $f_t$ towards $\unreg{[a,b]}$ as
$t$ goes to infinity, which is the aim of the following lemma.
\begin{lemm}
  \label{lemm:step_B}
  Let $0 \leq a \leq b$. For any $t > 0$ and $\lambda \in [0, + \infty)$,
  \begin{equation}
    \label{eq:step_B}
    \abs{f_t(\lambda) - \unreg{[a,b]}(\lambda)} \leq
    \left\{
      \begin{array}{ll}
        \decay(t \abs{\lambda -a}) & \text{if } \lambda \in [0,a) \cup \{b \} \\
        \decay(t \abs{\lambda -a}) + \decay(t \abs{\lambda -b}) & \text{if } \lambda \in (a,b) \\
        \decay(t \abs{\lambda -b}) & \text{if } \lambda \in \{ a \} \cup (b, +\infty)
      \end{array}
    \right.
  \end{equation}
  where $\decay : (0,+ \infty) \rightarrow \R$ is the (decreasing) function defined by
  $\decay(\rho) = \frac{e^{- \rho^2}}{2 \sqrt{\pi} \rho}$.
\end{lemm}

\begin{proof}
  Let us assume that $\lambda \in (b, + \infty)$.  In that case, $\unreg{[a,b]}(\lambda) = 0$, and
  \begin{equation*}
    \abs{f_t(\lambda) - \unreg{[a,b]}(\lambda)}
    = f_t(\lambda)
    \leq \frac{1}{\sqrt{\pi}} \int_{t(\lambda -b)}^{+ \infty} e^{- \rho^2} \d \rho
    \leq \frac{e^{-t^2(\lambda -b)^2}}{2 \sqrt{\pi} \; t(\lambda -b)} 
  \end{equation*}
  since for $\rho > t(\lambda -b)$, $1 < \frac{2\rho}{2 t(\lambda -b)}$. 
  All the other cases can be proved in the same way, using, when $\lambda \in [a,b]$, the fact that the Gaussian we used in
  the definition of $f_t$ is normalized.
\end{proof}

We can now prove the following estimate.

\begin{prop}
  \label{prop:integral_B}
  Let $0 \leq a \leq b$.
  For any $t >0$,
  \begin{equation*}
    \mathcal{R}_I(t,a,b) = \O{\frac{1}{t}}.
  \end{equation*}
  Whenever $b \leq \frac{1}{4}$, if $\bulkdist = \frac 1 4 - b$, 
  \begin{equation*}
    \mathcal{R}_I(t,a,b) = \O{\frac{e^{- \frac{3}{4} t^2 \bulkdist^2}}{t^{\frac 3 2}}}.
  \end{equation*}
\end{prop}

\begin{proof}
  Let us start with the case when $b \leq \frac 1 4$. Since $\tanh(x) \leq x$ for any $x \geq 0$,
  \begin{align*}
    \mathcal{R}_I(t,a,b)
    & = \O{\int_{\frac 1 4}^{+ \infty} \frac{\exp(-t^2(\lambda-b)^2)}{t \paren*{\lambda - b}} \; \sqrt{\lambda - \frac 1 4}
      \d \lambda} \\
    & = \O{\frac{1}{t^{\frac 3 2}} \int_{t \bulkdist}^{+ \infty} \frac{\exp(-u^2)}{\sqrt{u}} \d u} 
    = \O{\frac{\exp(-\frac 3 4 t^2 \bulkdist^2)}{t^{\frac 3 2}}}
  \end{align*}
  since $\frac{\exp(-\frac 1 4 u^2)}{\sqrt u}$ has finite integral on $(0,+\infty)$.
  
  In the general case, we can replace $\1{[a,b]}$ by the limit $\unreg{[a,b]}$ of $f_t$ as $t \rightarrow + \infty$, for
  they differ on a set of measure zero. We observe that the right hand side of equation~\eqref{eq:step_B} blows up around $a$ and
  $b$, while the left hand side remains bounded. We shall therefore cut small intervals around $a$ and $b$, and only
  apply Lemma \ref{lemm:step_B} outside them.

  Let $C_\epsilon$ be the set
  \begin{equation*}
    C_\epsilon = \left\{ \lambda \in \left[ \frac 1 4, + \infty \right) \; : \; \abs{\lambda - a} < \epsilon \text{ or }
      \abs{\lambda - b} < \epsilon  \right\}.
  \end{equation*}
  $C_\epsilon$ has at most two connected components, each of them of length at most $2 \epsilon$.   Since
  $\abs{f_t - \unreg{[a,b]}} \leq 1$,
  \begin{equation*}
    \int_{C_\epsilon} \abs{f_t(\lambda) - \unreg{[a,b]}(\lambda)} \d \lambda \leq 4 \epsilon. 
  \end{equation*}
  There are at most three connected components in $\left[ \frac 1 4 , + \infty \right) \setminus C_\epsilon$, and the estimate in every case is
  the same so we limit ourselves to the study of $[b + \epsilon, + \infty)$. Lemma~\ref{lemm:step_B} implies that
  \begin{equation*}
    \int_{b+\epsilon}^{+ \infty} \abs{f_t(\lambda) - \unreg{[a,b]}(\lambda)} \d \lambda
    = \O{ \int_{b+\epsilon}^{+ \infty} \frac{\exp \paren*{t^2(\lambda-b)^2}}{t(\lambda - b)}  \d \lambda}
    = \O{\frac{1}{\epsilon t^2} \int_{\epsilon t}^{+ \infty}  e^{-\rho^2} \d \rho}.
  \end{equation*}
  Putting the two contributions together, we obtain
  \begin{equation*}
    \mathcal{R}_I(t,a,b)
    = \O{\epsilon + \frac{1}{\epsilon t^2} \int_{\epsilon t}^{+ \infty}  e^{-\rho^2} \d \rho},
  \end{equation*}
  which leads to our claim if we set $\epsilon = \frac{1}{t}$.
\end{proof}

\subsection{The geometric term}
\label{sec:geometric_term_B}

Let us now control the geometric term
\begin{equation*}
\mathcal{R}_K(X,t,a,b) = \frac{1}{\volhyp_X (X)} \int_D \sum_{\gamma \in \Gamma \setminus \{\id\}}
K_t(z, \gamma \cdot z) \dvolhyp_{\mathcal{H}} (z)
\end{equation*}
for any compact hyperbolic surface $X \in \mathcal{A}_g$ from Corollary \ref{coro:geometry}.
In order to do so, we will estimate the kernel function $K_t$. We will then regroup the terms in the sum according to
the distance between $z$ and $\gamma \cdot z$. This is where we will use the Benjamini-Schramm hypothesis. Indeed, if
$z \in D$ has a large radius of injectivity, then the decay of $K_t$ will cause the sum to be small. Otherwise, the sum
might not be small, but the volume of the set of such $z$ will.

\subsubsection*{Fourier estimate}

In order to estimate the kernel $K_t$, we first need to know about the derivative of $g_t$, the inverse Fourier
transform of the test function $h_t$. 
\begin{lemm}
  \label{lemm:fourier_estimate}
  Let $0 \leq a \leq b \leq 1$
  and $r \in (0,\rmax)$.
  For any $u \geq r$, $t \geq \tminB$, 
  \begin{equation}
    \label{eq:estimate_g_prime_B}
    g_t'(u) = \O{r^{-\frac 2 3} \; \exp
       \paren*{-t^2 \bulkdist^2 -\frac{7}{32}  u^{\frac{4}{3}} t^{-\frac 2 3} + \frac{3}{16} u^{\frac{2}{3}} t^{\frac 2 3}}}
   \end{equation}
     where $\bulkdist = \max \paren*{\frac 1 4 - b, 0}$  is the distance between $[a,b]$ and $\left[ \frac 1 4, + \infty \right)$.
\end{lemm}
\begin{proof}
  By definition of $g_t$,
  \begin{equation*}
    g_t(u)
     = \frac{1}{2 \pi}  \int_{- \infty}^{+ \infty} h_t(r) e^{i r u} \d r \\
     = \frac{t}{2 \pi^{\frac{3}{2}}}  \int_{- \infty}^{+ \infty} \int_a^b e^{-t^2
      \paren*{\frac{1}{4} + r^2 - \mu}^2}  e^{i r u} \d \mu \d r.
  \end{equation*}
  By the change of variables $\tilde \mu = t \paren*{\mu - \frac 1 4}$ and $\tilde r = \sqrt t \; r$ and Fubini's theorem, one
  can rewrite this integral as
  \begin{equation*}
    g_t(u)
    = \frac{1}{2 \pi^{\frac{3}{2}} \sqrt t} \int_{t \paren*{a - \frac 1 4}}^{t \paren*{b - \frac 1 4}} \int_{- \infty}^{+ \infty}  e^{-
      \paren*{r^2 - \mu}^2}  e^{i \frac{r u}{\sqrt t}} \d r \d \mu.
  \end{equation*}
  As a consequence, the derivative of $g_t$ is
  \begin{equation}
    \label{eq:decomp_ft_elementary_B}
    g_t'(u) = \frac{i}{2 \pi^{\frac{3}{2}} t} \int_{t \paren*{a - \frac 1 4}}^{t \paren*{b - \frac 1 4}}
    F_{\frac{u}{\sqrt t}} \paren*{\mu} \d \mu
  \end{equation}
  where for $u >0$ and $\mu \in \R$,
  \begin{equation*}
    F_u(\mu) = \int_{- \infty}^{+ \infty}  r \; e^{-\paren*{r^2 - \mu}^2 + i r u} \d r.
  \end{equation*}
  Let us estimate this integral using a change of contour. Let $R >0$ and $R'>0$ be two real
  parameters. The function $z \in \C \mapsto z \; e^{- (z^2 - \mu)^2 +i z u}$ is holomorphic, so its contour integral
  on the rectangle of vertices $R'$, $R'+iR$, $-R' +iR$, $-R'$ is equal to zero.
  We compute the modulus of the integrand for a complex number $z = x + iy$:
  \begin{equation}
    \label{eq:modulus_exp_B}
    \abs{z \; e^{- (z^2 - \mu)^2 +i z u}}
    = \sqrt{x^2+y^2} \exp \paren*{- x^4 -y^4 +6 x^2 y^2 - \mu^2 + 2 x^2\mu - 2 y^2 \mu - yu}.
  \end{equation}
  It follows directly from this inequality that the integrals on the vertical sides of the rectangle go to zero as $R'$
  approaches infinity. As a consequence,
  \begin{equation*}
    F_u(\mu) = \int_{\R + i R}  z \; e^{-\paren*{z^2 - \mu}^2 + i z u} \d z.
  \end{equation*}
  We use the triangle inequality and equation~\eqref{eq:modulus_exp_B} to deduce that
  \begin{equation*}
    \abs{F_u(\mu)}
     \leq 2  \exp \paren*{- R^4 - \mu^2 - 2 R^2 \mu - u R} \int_{0}^{+ \infty} (x + R) \exp \paren*{-x^4 + 6x^2 R^2 +2x^2 \mu} \d x.
  \end{equation*}
  We now study two distinct cases, depending on the sign of $\mu$.
  \begin{itemize}
  \item If $\mu \geq 0$, then
    \begin{equation*}
      \abs{F_u(\mu)}
       \leq 2 \; \exp \paren*{8 R^4 + 4 R^2 \mu - u R} \int_{0}^{+ \infty} (x + R) \; e^{-(x^2 - x_0^2)^2} \d x
    \end{equation*}
    where $x_0 = \sqrt{3 R^2 + \mu} >0$. We observe that
    \begin{align*}
      \int_{0}^{+\infty}  (x+R) \; e^{-(x^2 - x_0^2)^2} \d x
      & \leq \int_{-x_0}^{+ \infty} (\abs{y} + x_0 + R) \; e^{-y^4} \underbrace{e^{-4 y^2 x_0 (y+x_0)}}_{\leq 1} \d y \\
      & = \O{1 + x_0 +R},
    \end{align*}
    and therefore
    \begin{equation*}
      F_u(\mu) = \O{(1 + R + \mu^{\frac 1 2}) \; \exp \paren*{8 R^4 + 4 R^2 \mu - u R} }.
    \end{equation*}
  \item If $\mu < 0$, we do the same with $x_0 = \sqrt{3} R$.
    \begin{align*}
      \abs{F_u(\mu)}
      & \leq 2 \exp \paren*{- \mu^2 + 8 R^4  - 2 R^2 \mu - u R} \int_{0}^{+ \infty} (x + R) \; e^{-(x^2 - 3 R^2)^2} \d x \\
      & = \O{(1+R) \exp \paren*{-\mu^2 + 8R^4 + 2 R^2 \abs{\mu} - u R}}.
    \end{align*}
  \end{itemize}
  As a conclusion, for any $\mu \in \R$,
  \begin{equation*}
    F_u(\mu) = \O{(1+R + \abs{\mu}^{\frac{1}{2}}) \exp \paren*{-\mu_-^2 + 8R^4 + 4 R^2 \abs{\mu} - u R}}
  \end{equation*}
  where $\mu_- = \min(\mu,0)$.
  We take $R = \frac{1}{4} \; u^{\frac{1}{3}}$ and obtain that, for any $u >0$ and $\mu \in \R$
  \begin{equation}
    \label{eq:estimate_elementary_part_F_B}
    F_u(\mu) = \O{(1 + u^{\frac{1}{3}} + \abs{\mu}^{\frac{1}{2}}) \; \exp \paren*{-\mu_-^2 -\frac{7}{32}  u^{\frac{4}{3}} +
        \frac{1}{4} \abs{\mu} u^{\frac{2}{3}}}}.
  \end{equation}
  We then integrate the upper bound~\eqref{eq:estimate_elementary_part_F_B} in
  equation~\eqref{eq:decomp_ft_elementary_B}.
  \begin{align*}
    g_t'(u)
    & = \O{\frac{1}{t} \int_{t \paren*{a - \frac 1 4}}^{t \paren*{b - \frac 1 4}}
    \abs{F_{\frac{u}{\sqrt t}} \paren*{\mu}} \d \mu} \\
    & = \O{\frac{1 + u^{\frac{1}{3}}t^{-\frac 1 6} + t^{\frac{1}{2}}}{t} \exp
      \paren*{-t^2 \bulkdist^2 -\frac{7}{32}  u^{\frac{4}{3}} t^{-\frac 2 3}} \int_{- \frac{3}{4} t}^{\frac{3}{4} t} \exp \paren*{\frac{1}{4} \abs{\mu} u^{\frac{2}{3}} t^{- \frac 1 3}} \d \mu}
  \end{align*}
  because $\abs{\mu} \leq \frac 3 4 t$ for any $\mu \in \brac*{t\paren*{a - \frac 1 4}, t \paren*{b - \frac 1 4}}$. 
  Replacing the integral by its value and using the hypotheses about $t$, $u$ and $r$ concludes the proof.
\end{proof}

\subsubsection*{Estimate of the kernel function}

Let us now estimate the kernel function $K_t$, using its expression in terms of $g_t'$, equation~\eqref{eq:selberg_transform}.

\begin{lemm}
  \label{lemm:estimate_K_B}
  Let $0 \leq a \leq b \leq 1$ and $r \in (0,\rmax)$.
  For any $\rho \geq r$, $t \geq \tminB$, 
  \begin{equation}
    \label{eq:estimate_K_B}
    K_t(\rho) =
    \left\{
      \begin{array}{ll}
        \O{\frac{t}{r^2} \exp \paren*{-t^2 \bulkdist^2 - \frac{\rho^{\frac 4 3}}{8 t^{\frac 2  3}}}} & \text{if } \rho \geq 6 t^2 \\
        \O{\frac{t}{r^2}  \exp \paren*{-t^2 \bulkdist^2 + t^2 - \frac{\rho^{\frac 4 3}}{8 t^{\frac 2 3}}}} & \text{if } \rho \leq 6 t^2
      \end{array}
    \right.
  \end{equation}
    where $\bulkdist = \max \paren*{\frac 1 4 - b, 0}$  is the distance between $[a,b]$ and $\left[ \frac 1 4, + \infty \right)$.
\end{lemm}

\begin{proof}
  By definition of the kernel associated to $h_t$,
  \begin{equation*}
    K_t(\rho) = - \frac{1}{\sqrt 2 \pi} \int_\rho^{+ \infty} \frac{g_t'(u)}{\sqrt{\cosh u - \cosh \rho}} \d u.
  \end{equation*}
  Using equation~\eqref{eq:estimate_g_prime_B}, we obtain
  \begin{equation}
    \label{eq:k_estime_g_B}
    K_t(\rho) = \O{\frac{\exp \paren*{-t^2 \bulkdist^2}}{r^{\frac 2 3}}
      \int_\rho^{+ \infty}  \underbrace{\frac{\exp \paren*{-\frac{7}{32}  u^{\frac{4}{3}} t^{-\frac 2 3} + \frac{3}{16}
             u^{\frac{2}{3}} t^{\frac 2 3}}}{\sqrt{\cosh u - \cosh \rho}}}_{(\star)} \d u}.
  \end{equation}
  Let us cut this integral into two contributions $I_1 = \int_\rho^{\rhocut} (\star)$ and $I_2 = \int_{\rhocut}^{+ \infty}(\star)$
  where
  \begin{equation*}
    \rhocut = \max(2\rho,12 t^2).
  \end{equation*}
  This choice of $\rhocut$ allows us deal with the cancellation of the denominator in $I_1$ only, and to be in the
  asymptotic regime where the decreasing part of the exponential term is predominant everywhere in $I_2$.

  In the second integral, since $u \geq \rhocut \geq 12 t^2$,
  $\frac{3}{16} u^{\frac 2 3} t^{\frac 2 3} < \frac{2}{32} u^{\frac 4 3} t^{- \frac 2 3}$.  Hence, the quantity in the
  exponential function is less than $-\frac{5}{32} u^{\frac{4}{3}} t^{-\frac 2 3}$.  We deal with the denominator by
  observing that $\cosh u - \cosh \rho \geq \frac{1}{2} (u-\rho)^2 \geq \frac{1}{2} \rho^2$ since
  $u \geq \rhocut \geq 2 \rho$. As a consequence,
  \begin{equation*}
    I_2 = \O{\frac{1}{\rho} \int_{\rhocut}^{\infty} \exp \paren*{- \frac{5}{32} u^{\frac{4}{3}} t^{-\frac 2 3}} \d u}.
  \end{equation*}
  This integral can be controlled by observing that $1 \leq u^{\frac 1 3} \rhocut^{-\frac 1 3}$, and then by explicit
  integration:
  \begin{equation}
        \label{eq:i_two_B}
    I_2 = \O{\frac{t^{\frac 2 3}}{\rho \rhocut^{\frac 1 3}} \exp \paren*{- \frac{5}{32} \rhocut^{\frac{4}{3}} t^{-\frac 2 3}}}
    = \O{\frac{t^{\frac 2 3}}{r^{\frac 4 3}} \exp \paren*{- \frac{5}{16} \rho^{\frac{4}{3}} t^{-\frac 2 3}}}.
  \end{equation}
  
  Now, in the first integral we use the inequality $\cosh u - \cosh \rho \geq (u-\rho) \sinh \rho \geq (u-\rho)\rho$.
  \begin{align*}
    I_1
    & \leq 
     \frac{\exp \paren*{-\frac{7}{32}  \rho^{\frac{4}{3}} t^{-\frac 2 3} + \frac{3}{16}  \rhocut^{\frac{2}{3}}
      t^{\frac 2 3}}}{\sqrt \rho} \int_\rho^{\rhocut}  \frac{\d u}{\sqrt{u - \rho}}  \\
    & = \frac 1 2 \sqrt{\frac{\rhocut-\rho}{\rho}} \exp \paren*{-\frac{7}{32}  \rho^{\frac{4}{3}} t^{-\frac 2 3} + \frac{3}{16}  \rhocut^{\frac{2}{3}}
      t^{\frac 2 3}}.
  \end{align*}
  \begin{itemize}
  \item When $\rho \leq 6 t^2$, $\rhocut = 12 t^2$ so 
    \begin{equation}
      \label{eq:i_one_small_rho_B}
      I_1 = \O{\frac{t}{r^{\frac 1 2}} \exp \paren*{- \frac{7}{32} \rho^{\frac 4 3}t^{- \frac 2 3} + t^2}}.
    \end{equation}
  \item   Otherwise, $\rhocut = 2\rho$ so
    \begin{equation}
      \label{eq:i_one_large_rho_B}
      I_1
      = \O{\exp \paren*{-\frac{7}{32}  \rho^{\frac{4}{3}} t^{-\frac 2 3} + \frac{3}{16}  2^{\frac 2 3}  \rho^{\frac{2}{3}}
          t^{\frac 2 3}}}
      = \O{\exp \paren*{-\frac{1}{8}  \rho^{\frac{4}{3}} t^{-\frac 2 3}}}
    \end{equation}
    because the fact that $\rho \geq 6 t^2$ implies that
    $\frac{3}{16} 2^{\frac 2 3} \rho^{\frac 2 3} t^{\frac 2 3} < \frac{3}{32} \rho^{\frac 4 3} t^{- \frac 2 3}$.
  \end{itemize}
  Putting together~\eqref{eq:k_estime_g_B}, \eqref{eq:i_two_B},
  \eqref{eq:i_one_small_rho_B} and \eqref{eq:i_one_large_rho_B} leads to what was
  claimed.
\end{proof}

\subsubsection*{Kernel summation}

We can now proceed to the estimate of the geometric term. In order to do so, we will arrange the terms in the sum
depending on the distance between $z$ and $\gamma \cdot z$, and use Lemma~\ref{lemm:estimate_K_B}.

\begin{lemm}
  \label{lemm:kernel_sum_B}
  Let $0 \leq a \leq b \leq 1$.  Let $r \in (0,\rmax)$ and $X$ be a compact hyperbolic surface
  of radius of injectivity larger than $r$.  For any $t \geq \tminB$, $L \geq 2^{12} t^2$,
  \begin{equation}
    \label{eq:kernel_sum_B}
    \mathcal{R}_K(X,t,a,b) =
    \O{\frac{t^3}{r^4} \; \exp \paren*{-t^2 \bulkdist^2} \; 
      \brac*{\exp \paren*{-L}
        + \frac{\volhyp_X (X^-(L))}{\volhyp_X (X)} \; \exp \paren*{L} }}
  \end{equation}
  where $\bulkdist = \max \paren*{\frac 1 4 - b, 0}$ is the distance between $[a,b]$ and
  $\left[ \frac 1 4, + \infty \right)$ and $X^-(L)$ is the set of points in $X$ of radius of injectivity smaller than
  $L$.
\end{lemm}

\begin{proof}
  Let us write a fundamental domain $D$ of $X = \faktor{\mathcal{H}}{\Gamma}$ as a disjoint union of $D^+(L)$ and
  $D^-(L)$, respectively the set of points in $D$ of radius of injectivity larger and smaller than $L$.  We cut
  according to this partition of $D$ the integral in the definition of
  \begin{equation*}
    \mathcal{R}_K(X,t,a,b) = \frac{1}{\volhyp_X (X)} \int_D \sum_{\gamma \neq \id} K_t(z, \gamma \cdot z) \dvolhyp_{\mathcal{H}}(z)
  \end{equation*}
  into two contributions, $\mathcal{R}_K^+(X,t,a,b,L)$ and $\mathcal{R}_K^-(X,t,a,b,L)$.

  Let us first estimate the term $\mathcal{R}_K^+(X,t,a,b,L)$.  Lemma~\ref{lemm:estimate_K_B} allows us to control
  $K_t(z, \gamma \cdot z)$ in terms of the distance between $z$ and $\gamma \cdot z$. In order to use it, we regroup
  the terms of the sum according to this quantity:
  \begin{equation*}
    \mathcal{R}_K^+(X,t,a,b,L)
    = \frac{1}{\volhyp_X (X)} \int_{D^+(L)} \sum_{j \geq  L} \sum_{\substack{\gamma
        \neq \id \\ j \leq \dist_{\mathcal{H}}(z, \gamma \cdot z) < j+1}} K_t(z, \gamma \cdot z) \dvolhyp_{\mathcal{H}}(z).
  \end{equation*}
  One should notice that the sum only runs over integers larger than or equal to $L$ as a consequence of the definition
  of $D^+(L)$.  For any $z \in D^+(L)$, $j \geq L$ and $\gamma \in \Gamma \setminus \{ \id \}$ such that
  $j \leq \dist_{\mathcal{H}}(z, \gamma \cdot z) < j+1$, by Lemma~\ref{lemm:estimate_K_B} and since
  $\dist_{\mathcal{H}}(z, \gamma \cdot z) \geq L > 6 t^2$,
  \begin{equation*}
    K_t(z, \gamma \cdot z) = \O{\frac{t}{r^2} \exp \paren*{-t^2 \bulkdist^2 -\frac{j^{\frac 4 3}}{8 t^{\frac 2 3}} }}.
  \end{equation*}
  We then apply the following lemma, inspired by~\cite{lemasson2017}, to control the number of
  $\gamma \in \Gamma$ contributing to the sum.
  \begin{lemm}
    \label{lemm:counting_gamma}
    Let $r \in (0,\rmax)$ and $X = \faktor{\mathcal{H}}{\Gamma}$ be a compact hyperbolic
    surface of radius of injectivity larger than $r$. Then, for any $z \in \mathcal{H}$
    and any $j >0$,
    \begin{equation}
      \label{eq:counting_gamma}
      \#\{ \gamma \in \Gamma \; : \; \dist_{\mathcal{H}}(z, \gamma \cdot z) \leq j \} = \O{\frac{e^j}{r^2}}.
    \end{equation}
  \end{lemm}
  \begin{proof}
    By definition of the radius of injectivity, the balls $B_\gamma$ of center $\gamma \cdot z$ and radius $\frac r 2$,
    for $\gamma \in \Gamma$, are disjoint. If $\gamma$ is such that $\dist_{\mathcal{H}}(z, \gamma \cdot z) \leq j$,
    then $B_\gamma$ is included in the ball of center $z$ and radius $j+\frac r 2$. Since the volume of a hyperbolic
    ball of radius $R$ is $\cosh(R)-1$, the number of such $\gamma$ is smaller than
    \begin{equation*}
      \frac{\cosh \paren*{j+\frac r 2} -1}{\cosh \paren*{\frac r 2} - 1}
      = \O{\frac{e^{j}}{r^2}}.
    \end{equation*}
  \end{proof}
  Therefore, and because $\volhyp_{\mathcal{H}}(D^+(L)) \leq \volhyp_X(X)$,
  \begin{equation*}
    \mathcal{R}_K^+(X,t,a,b,L) = \O{\frac{t}{r^4}\; \exp \paren*{-t^2 \bulkdist^2} \; S(t,L)}
  \end{equation*}
  where $S(t,L)$ is defined as the sum
  \begin{equation*}
    S(t,L) := \sum_{j \geq L} \exp \paren*{j -\frac{j^{\frac 4 3}}{8 t^{\frac 2 3}}}.
  \end{equation*}
  The fact that $L \geq 2^{12} t^2$ implies that, for any $j \geq L$, $j \leq \frac{j^{\frac 4 3}}{16 t^{\frac 2 3}}$.
  As a consequence,
  \begin{equation*}
    S(t,L) \leq \sum_{j \geq L} \exp \paren*{-\frac{j^{\frac 4 3}}{16 t^{\frac 2 3}}}
  \end{equation*}
  which can be estimated by comparison with an integral:
  \begin{equation*}
    S(t,L) 
     \leq \paren*{1 + \frac{12  t^{\frac 2 3}}{L^{\frac 1 3}}} \exp \paren*{-\frac{L^{\frac 4 3}}{16
       t^{\frac 2 3}}} 
     = \O{\exp \paren*{-L}} \quad \text{since } L^{\frac 1 3} \geq 16 t^{\frac 2 3}.
  \end{equation*}
  Therefore,
  \begin{equation}
    \label{eq:rplus_B}
    \mathcal{R}_K^+(X,t,a,b,L) = \O{\frac{t}{r^4} \; \exp \paren*{-t^2 \bulkdist^2} \;\exp \paren*{-L} }.
  \end{equation}

  The same method, applied to $\mathcal{R}_K^-(X,t,a,b,L)$, leads to
  \begin{equation*}
    \mathcal{R}_K^-(X,t,a,b,L) = \O{\frac{t}{r^4}\exp \paren*{-t^2
        \bulkdist^2}  \frac{\volhyp_X (X^-(L))}{\volhyp_X (X)} \sum_{j \geq
        0}(1+\1{[0,6t^2]}(j) \; e^{t^2}) \exp \paren*{j -\frac{j^{\frac 4 3}}{8 t^{\frac 2 3}}}}.
  \end{equation*}
  We cut the sum at $\jcut^{(1)} = \lfloor 6 t^2 \rfloor +1$ and $\jcut^{(2)} = \lfloor 2^{12} t^2 \rfloor +1$. The term
  where $j \geq \jcut^{(2)}$ satisfies the same estimate as before since $\jcut^{(2)} \geq 2^{12}t^2$, and therefore is
  $\O{\exp(-\jcut^{(2)})} = \O{1}$.  We control naively the two other terms, which are $\O{t^2 \exp(2^{12}t^2)}$.  As a
  consequence,
  \begin{equation}
    \label{eq:rminus_B}
    \mathcal{R}_K^-(X,t,a,b,L) = \O{\frac{t^3}{r^4} \; \exp \paren*{-t^2 \bulkdist^2} \; \frac{\volhyp_X (X^-(L))}{\volhyp_X (X)} \; \exp(2^{12}t^2) }.
  \end{equation}
  Our claim follows directly from equations~\eqref{eq:rplus_B} and \eqref{eq:rminus_B}.
\end{proof}

\subsubsection*{Geometric estimate}

We are now ready to use the geometric properties of random hyperbolic surfaces, and obtain an estimate of the geometric
term in the trace formula true \emph{with high probability}.
\begin{prop}
  \label{prop:geometric_term_B}
  For any large enough $g$, any $0 \leq a \leq b \leq 1$ and any hyperbolic surface $X \in \mathcal{A}_g$ defined in
  Corollary~\ref{coro:geometry}, if we set $t = \frac{\sqrt{\log g}}{64 \sqrt{6}}$, then
  \begin{equation}
    \label{eq:kernel_B}
    \mathcal{R}_K(X,t,a,b) = \O{\frac{g^{-\frac{1}{3 \cdot 2^{13}} \bulkdist^2}}{(\log g)^{\frac 3 4}}}
  \end{equation}
  where $\bulkdist = \max \paren*{\frac 1 4 -b, 0}$ is the distance between $[a,b]$ and $\left[ \frac 1 4, + \infty \right)$.
\end{prop}

\begin{proof}
  It is a direct consequence of Lemma~\ref{lemm:kernel_sum_B} and the properties of the elements of $\mathcal{A}_g$,
  namely that if $X$ is an element of $\mathcal{A}_g$ and $L = \frac 1 6 \log g = 2^{12}t^2$, then
  \begin{itemize}
  \item the injectivity radius of $X$ is greater than $r = g^{-\frac{1}{24}}
    (\log g)^{\frac{9}{16}}$;
  \item $\dfrac{\volhyp_X(X^-(L))}{\volhyp_X(X)} = \O{g^{-\frac{1}{3}}}$.
  \end{itemize}
  Since $L \geq 2^{12} t^2$, $r < \rmax$ and $t > \tminB$, we can apply
  Lemma~\ref{lemm:kernel_sum_B}:
  \begin{align*}
    \mathcal{R}_K(X,t,a,b)
    & = \O{\frac{t^3}{r^4} \; \exp \paren*{-t^2 \bulkdist^2} \; 
      \brac*{\exp \paren*{-L}
      + \frac{\volhyp_X (X^-(L))}{\volhyp_X (X)} \; \exp \paren*{L} }} \\
    & = \O{\frac{(\log g)^{\frac 3 2}}{g^{-\frac{1}{6}} (\log g)^{\frac 9 4}} \; g^{-\frac{1}{3 \cdot 2^{13}} \bulkdist^2}
      \brac*{g^{-\frac 1 6} + g^{- \frac 1 3 + \frac 1 6}}}.
  \end{align*}
\end{proof}

\subsection{Proof of Theorem~\ref{theo:upper_bound} at the bottom of the spectrum}
\label{sec:upper_bound_B}

When we put together equation~\eqref{eq:trace_applied_B}, Proposition \ref{prop:integral_B}
and~\ref{prop:geometric_term_B}, we obtain directly the following statement, which is an estimate of the trace
formula.  Theorems \ref{theo:upper_bound} and \ref{theo:equivalent} (when $b \leq 1$) will then follow, since the
spectral sum $\sum_{j=0}^{+ \infty} f_t(\lambda_j)$ approaches $\counting{\Delta}{X}{a,b}$ as $t \rightarrow + \infty$.

\begin{coro}
  \label{coro:trace_formula_estimate_B}
  For any large enough $g$, any $0 \leq a \leq b \leq 1$ and any hyperbolic surface $X \in \mathcal{A}_g$ defined in
  Corollary \ref{coro:geometry}, if we set $t = \frac{\sqrt{\log g}}{64 \sqrt{6}}$, then
  \begin{equation}
    \label{eq:trace_formula_estimate_B}
    \frac{1}{\volhyp_X(X)} \sum_{j=0}^{+ \infty} f_t(\lambda_j)
    = \integralterm + \O{\frac{1}{\sqrt{\log g}}}.
  \end{equation}
  If $b \leq \frac 1 4$, then we also have, for $c = 2^{-15}$,
  \begin{equation}
    \label{eq:trace_formula_estimate_B_small}
    \frac{1}{\volhyp_X(X)} \sum_{j=0}^{+ \infty} f_t(\lambda_j)
    = \O{\frac{g^{-c \paren*{\frac{1}{4}-b}^2}}{\paren*{\log g}^{\frac 3 4}}}.
  \end{equation}
\end{coro}
\begin{proof}[Proof of Theorem~\ref{theo:upper_bound} when $b \leq 1$]
  Let $t = \frac{\sqrt{\log g}}{64 \sqrt{6}}$. 
  Let us distinguish two cases.
  \begin{itemize}
  \item Whenever $t(b-a) \geq \frac{1}{\sqrt 3}$, since the function $f_t$ only takes positive values,
    \begin{equation*}
    \frac{\counting{\Delta}{X}{a,b}}{\volhyp_X(X)} \times \inf_{[a,b]} f_t \leq  \frac{1}{\volhyp_X(X)} \sum_{j=0}^{+ \infty} f_t(\lambda_j).
  \end{equation*}
  It follows directly from equation \eqref{eq:trace_formula_estimate_B} that the right hand term is
  \begin{equation*}
    \O{b-a + \frac{1}{\sqrt{\log  g}}}.
  \end{equation*}
 In order to deal with the infimum, we use Lemma \ref{lemm:step_B} that states that
  \begin{equation*}
     \inf_{[a,b]} f_t \geq \frac{1}{2} - \frac{e^{-t^2(b-a)^2}}{2 \sqrt{\pi} t(b-a)} 
     \geq \frac{1}{2} - \frac{\sqrt 3 \; e^{-\frac 1 3}}{2 \sqrt{\pi}} \geq \frac{1}{10} 
     \end{equation*}
     since we assumed $t(b-a) \geq \frac{1}{\sqrt 3}$. Therefore, $\paren*{\inf_{[a,b]} f_t}^{-1} = \O{1}$.
   \item Otherwise, the fact that $a$ and $b$ are very close together prevents the test function $f_t$ from being a good
     approximation of the indicator function of $[a,b]$. We therefore let $a'$ be $b - \frac{1}{\sqrt 3 \; t}$, so that
     $a'$ and $b$ satisfy the spacing hypothesis $t(b-a') \geq \frac{1}{\sqrt 3}$, and we can apply the first point to
     them:
     \begin{equation*}
       \frac{\counting{\Delta}{X}{a,b}}{\volhyp_X(X)}
       \leq  \frac{\counting{\Delta}{X}{a',b}}{\volhyp_X(X)}
       = \O{b-a' + \frac{1}{\sqrt{\log g}}}
       = \O{\frac{1}{\sqrt{\log g}}}.
     \end{equation*}
     The issue with this fix is that, when $b$ is small, $a'$ takes negative values. However, throughout this section,
     the only place where the positivity of $a'$ was used is at the end of Lemma \ref{lemm:fourier_estimate}, when
     saying that $\abs{\mu} \leq \frac 3 4 t$ for any
     $\mu \in \left[ t \paren*{a'-\frac 1 4}, t \paren*{b-\frac 1 4} \right]$. This remains true as soon as $a' \geq -
     \frac 1 2$, which will be the case if $t$ is large enough.
   \end{itemize}
   
   The other proof is the same, using the small eigenvalue case of Corollary~\ref{coro:trace_formula_estimate_B}.
 \end{proof}

\subsection{Proof of Theorem \ref{theo:equivalent} at the bottom of the spectrum}
\label{sec:comparison_B}

Let us now proceed to the proof of Theorem~\ref{theo:equivalent} when $b \leq 1$. Beware that, in the proof of the
lower bound, we will need to use Theorem~\ref{theo:upper_bound} \emph{for any $0 \leq a \leq b$}. This is not an issue,
as was shown in Figure~\ref{fig:dependencies}.

\begin{proof}[Proof of the upper bound of Theorem~\ref{theo:equivalent} when $b \leq 1$]
    Let $t = \frac{\sqrt{\log g}}{64 \sqrt{6}}$.
  
  If $t(b-a) \leq \sqrt{2e}$, then the integral term is $\O{b-a} = \O{\frac{1}{\sqrt{\log g}}}$, so the result
  follows directly from Theorem \ref{theo:upper_bound}. 
  
  Let us assume $t(b-a) \geq \sqrt{2e}$.  The issue in the previous estimate was that the convergence of $f_t$ is slow
  around $a$ and $b$, and noticeably $f_t(a)$ and $f_t(b)$ go to $\frac 1 2$ and not $1$ as $t \rightarrow + \infty$. In
  order to deal with this, we cut a small segment around $a$ and $b$. Let $\frac 1 t \leq \epsilon \leq \frac{b-a}{2}$,
  then
  \begin{equation*}
    \counting{\Delta}{X}{a,b} = \counting{\Delta}{X}{a,a+\epsilon} + \counting{\Delta}{X}{a+\epsilon,b-\epsilon} + \counting{\Delta}{X}{b-\epsilon,b}.
  \end{equation*}
  By Theorem \ref{theo:upper_bound}, 
  \begin{equation*}
    \frac{\counting{\Delta}{X}{a,a+\epsilon} + \counting{\Delta}{X}{b-\epsilon,b}}{\volhyp_X(X)}
    = \O{\epsilon + \sqrt{\frac{b+1}{\log g}}}
    = \O{\epsilon}.
  \end{equation*}
  We use the same method as before to control the middle term:
  \begin{align*}
    \frac{\counting{\Delta}{X}{a+\epsilon,b-\epsilon}}{\volhyp_X(X)} \times \inf_{[a+\epsilon,b-\epsilon]} f_t
    & \leq \frac{1}{\volhyp_X(X)} \sum_{j=0}^{+ \infty} f_t(\lambda_j) \\
    & \leq \integralterm + \frac{C'}{\sqrt{\log g}} 
  \end{align*}
  for a constant $C'>0$, given by Corollary \ref{coro:trace_formula_estimate_B}.
  By Lemma \ref{lemm:step_B}, and because $\epsilon t \geq 1$,
  \begin{equation*}
    \inf_{[a+\epsilon,b-\epsilon]} f_t \geq 1 - \frac{e^{-t^2 \epsilon^2}}{2 \sqrt \pi \epsilon t}  \geq  \frac{1}{1 + e^{- \epsilon^2 t^2}} \cdot
  \end{equation*}
  Putting all the contributions together, there exists a constant $C''>0$ such that
  \begin{equation*}
    \frac{\counting{\Delta}{X}{a,b}}{\volhyp_X(X)} \leq \integralterm + C'' \; \paren*{\epsilon 
      + (b-a) \; e^{-t^2\epsilon^2}}.
  \end{equation*}
  We can now set
  \begin{equation*}
    \epsilon = \frac{1}{t} \sqrt{\log \paren*{\frac{\sqrt e \; t (b-a)}{\sqrt 2}}}.
  \end{equation*}
  The hypothesis $t(b-a) \geq \sqrt{2 e}$ directly implies that $\epsilon t \geq 1$.  Furthermore, the fact that for any
  $x \geq 1$, $\sqrt{\log x} \leq \frac{x}{\sqrt{2 e}}$ implies that $\epsilon \leq \frac{b-a}{2}$.
  Direct substitution of $\epsilon$ and $t$ by their values in the previous estimate leads to our claim.
\end{proof}

 \begin{proof}[Proof of the lower bound of Theorem \ref{theo:equivalent} when $b \leq 1$]
     Let $t = \frac{\sqrt{\log g}}{64 \sqrt 6}$. Since $0 \leq f_t \leq 1$,
   \begin{equation*}
     \counting{\Delta}{X}{a,b}
     \geq \sum_{a \leq \lambda_j \leq b} f_t(\lambda_j)
     = \sum_{j=0}^{+ \infty} f_t(\lambda_j)
     - \sum_{0 \leq \lambda_j < a} f_t(\lambda_j)
     - \sum_{\lambda_j > b} f_t(\lambda_j).
   \end{equation*}
   By Corollary \ref{coro:trace_formula_estimate_B}, there exists a $C'>0$ such that
   \begin{equation*}
     \frac{1}{\volhyp_X(X)} \sum_{j=0}^{+ \infty} f_t(\lambda_j)
    \geq \integralterm -  \frac{C'}{\sqrt{\log g}},
  \end{equation*}
  so it suffices to prove that the two remaining terms are $\O{\frac{\volhyp_X(X)}{\sqrt{\log g}}}$.

  Both the terms behave the same way, so we only detail the sum over $b$.
  Let us divide $(b, + \infty)$ using a subdivision $b_k = b + \frac{k}{t}$, $k \geq 0$.
  We regroup the terms of the sum according to these numbers.
  \begin{align*}
    \frac{1}{\volhyp_X(X)} \sum_{\lambda_j > b} f_t(\lambda_j)
    & = \sum_{k=0}^{+ \infty} \frac{1}{\volhyp_X(X)}  \sum_{b_k < \lambda_j \leq b_{k+1}} f_t(\lambda_j) \\
    & \leq \sum_{k=0}^{+ \infty}  \frac{\counting{\Delta}{X}{b_k,b_{k+1}}}{\volhyp_X(X)} \times \sup_{[b_k,b_{k+1}]} f_t \\
    & = \O{\sum_{k=0}^{+ \infty} \paren*{b_{k+1}-b_k +
      \sqrt{\frac{b_{k+1}+1}{\log g}}}  \times \sup_{[b_k,b_{k+1}]} f_t} 
  \end{align*}
  by Theorem~\ref{theo:upper_bound}. As a consequence,
  \begin{equation*}
    \frac{1}{\volhyp_X(X)} \sum_{\lambda_j > b} f_t(\lambda_j)
     = \O{\frac{1}{\sqrt{\log g}} + \frac{1}{\sqrt{\log g}} \sum_{k=1}^{+
         \infty} \sqrt k \times \sup_{[b_k,b_{k+1}]} f_t}.
   \end{equation*}
   By Lemma~\ref{lemm:step_B},
   \begin{equation*}
     \sum_{k=1}^{+\infty} \sqrt k \times \sup_{[b_k,b_{k+1}]} f_t
     = \O{\sum_{k=1}^{+ \infty} \frac{\exp \paren*{-k^2}}{\sqrt k}}
     = \O{1}.
   \end{equation*}
 \end{proof}

\section{Proof of Theorems~\ref{theo:upper_bound} and \ref{theo:equivalent}  away from small eigenvalues}
\label{sec:H}

We now proceed to the proof of Theorems~\ref{theo:upper_bound} and \ref{theo:equivalent} in the case when $\frac 1 2
\leq a \leq b$. It is easy to see that it suffices to prove the results for these two situations, for we can then apply
them to $a, \frac 3 4$ and $\frac 3 4, b$ and add up the two contributions if $a < \frac 1 2$ and $b > 1$.

The proof here is very similar to the previous proof, apart from the fact that the test function we use is different. We
will not give all the details, and mostly highlight the differences between the two proofs.

The reason why we need to assume $a \geq \frac 1 2$ is that the test function we use behaves poorly for small
eigenvalues. We will use the fact that there are at most $2g-2$ of them by work of Otal and Rosas~\cite{otal2009}, and
that the spectral window is far enough from them, to deal with this situation (see Section \ref{sec:upper_bound_H}).

\subsection{Trace formula, test function and sketch of the proof}
\label{sec:test_function_H}

We will use once again the Selberg trace formula, but with a different test function this time.

\subsubsection*{The test function}

Let $a = \frac 1 4 + \alpha^2$ and $b = \frac 1 4 + \beta^2$, for some $0
\leq \alpha \leq \beta$. 
Let us consider the function
\begin{equation*}
  h_t(r) = (\1{[\alpha,\beta]} \star v_t)\paren*{r}
  = \frac{t}{\sqrt \pi} \int_\alpha^\beta \exp \paren*{-t^2(r-\rho)^2} \d \rho
  =\frac{1}{\sqrt \pi} \int_{t(\alpha - r)}^{t(\beta - r)} \exp \paren*{- \rho^2} \d \rho,
\end{equation*}
where $t$ still grows like $\sqrt{\log g}$. $h_t$ now is a smooth approximation of the
function $\unreg{[\alpha,\beta]}$.
We make $h_t$ into an even test function by setting $H_t(r) = h_t(r) + h_t(-r)$.
It is clear that $H_t : \C \rightarrow \C$ is analytic and even. The following lemma is an estimate on $h_t$ aimed at
applying the trace formula, but we make it a bit more precise than necessary for later use.
\begin{lemm}
  \label{lemm:trace_decrease_H}
  Let $0 \leq \alpha \leq \beta$, $a = \frac 1 4 + \alpha^2$ and $b = \frac 1 4 + \beta^2$. For any $t > 0$,
  \begin{equation*}
    \forall r = x + iy, \quad \abs{h_t(r)} \leq  \frac{1}{2 \sqrt \pi \; \alpha t} \exp \paren*{t^2(y^2-x^2 + 2 \beta x - \alpha^2)}.
  \end{equation*}
\end{lemm}
\begin{proof}
  Let $r = x+iy$. The modulus of the integrand in the definition of $h_t(r)$ is
  \begin{equation*}
    \abs{\exp \paren*{-t^2(r-\rho)^2}}
     = \exp \paren*{-t^2 (x-\rho)^2 + t^2y^2}.
   \end{equation*}
   As a consequence,
   \begin{equation*}
     \abs{h_t(r)}
      \leq \frac{t}{\sqrt \pi} \exp \paren*{t^2(y^2-x^2 + 2 \beta x)} \int_\alpha
       ^{+\infty} \exp \paren*{-t^2 \rho^2} \d \rho 
     \end{equation*}
     which allows us to conclude, using the Gaussian tail estimate.
\end{proof}

Therefore, one can apply the trace formula to $H_t$:
\begin{equation}
  \label{eq:trace_applied_H}
    \frac{1}{\volhyp_X (X)} \sum_{j=0}^{+ \infty} (h_t(r_j) + h_r(-r_j))
  = \integraltermH
  + \mathcal{R}_I(t,a,b) + \mathcal{R}_K(X,t,a,b)
\end{equation}
where
\begin{align}
  \mathcal{R}_I(t,a,b)
  & = \frac{1}{4 \pi} \int_{0}^{+ \infty} (h_t(r) + h_t(-r) - \1{[\alpha,\beta]}(r))
    \; r \tanh \paren*{\pi r}  \d r, \\
  \mathcal{R}_K(X,t,a,b)
  & = \frac{1}{\volhyp_X (X)} \int_D \sum_{\gamma \in \Gamma \setminus \{\id\}} K_t(z, \gamma \cdot z) \dvolhyp_{\mathcal{H}} (z),
\end{align}
$K_t$ is the kernel associated to $H_t$ and $D$ is a fundamental domain of $X = \faktor{\mathcal{H}}{\Gamma}$.

\subsubsection*{Sketch of the proof}

The steps of the proof are exactly the same as before, and are organized the same way. The only additional step is
dealing with the contributions of the small eigenvalues to the sum $\sum_{j=0}^{+ \infty} (h_t(r_j) + h_r(-r_j))$, and
can be found in Section~\ref{sec:upper_bound_H}. This is necessary here and was not before because the function $h_t$ is
no longer real valued and small on the imaginary axis. This complication is the reason why this test function does not
work whenever $a < \frac 1 2$.

\subsection{The integral term}
\label{sec:integral_term_H}

The integral estimate is the following.

\begin{prop}
  \label{prop:integral_H}
  Let $\frac 1 4 \leq a \leq b$.
  For any $t \geq \tminH$,
  \begin{equation*}
    \mathcal{R}_I(t,a,b) = \O{\frac{\sqrt{b}}{t}}.
  \end{equation*}
\end{prop}

The proof uses the same method as before, and the following lemma to control the speed of convergence of $h_t$
towards $\unreg{[\alpha,\beta]}$ as $t$ goes to infinity.
\begin{lemm}
  \label{lemm:step_H}
  Let $0 \leq \alpha \leq \beta$.
  For any $t > 0$ and $r  \in \R$,
  \begin{equation}
    \label{eq:step_H}
    \abs{h_t(r) - \unreg{[\alpha,\beta]}(r)} \leq
    \left\{
      \begin{array}{ll}
        \decay(t \abs{r - \alpha}) & \text{if } r \in (- \infty, \alpha) \cup \{\beta \} \\
        \decay(t \abs{r -\alpha}) + \decay(t \abs{r -\beta}) & \text{if } r \in (\alpha,\beta) \\
        \decay(t \abs{r -\beta}) & \text{if } r \in \{ \alpha \} \cup (\beta, +\infty)
      \end{array}
    \right.
  \end{equation}
  where $\decay : (0,+ \infty) \rightarrow \R$ is the (decreasing) function defined in Lemma~\ref{lemm:step_B}.
\end{lemm}
    
\subsection{The geometric term}
\label{sec:geometric_term_H}

The control of the geometric term is simpler in this case, because the test function $H_t$ is a convolution of two
functions with simple Fourier transforms (a Gaussian and a step function). Therefore, its Fourier transform has a simple
expression.

\begin{lemm}
  \label{lemm:g_derivative_H}
  Let $\frac 1 4 \leq a \leq b$ and $r \in (0,\rmax)$. 
  For any $t\geq \tminH$, $u>r$,
  \begin{equation}
    g_t'(u)
     = \O{\frac{\sqrt b}{r} \exp \paren*{- \frac{u^2}{4t^2}}}.
  \end{equation}
\end{lemm}

\begin{proof}
    Let us write $a = \frac 1 4 + \alpha^2$ and $b = \frac 1 4 + \beta^2$, with
  $0 \leq \alpha \leq \beta$.

  We can compute $g_t$ explicitly, knowing the Fourier transform of a Gaussian and a step function:
   \begin{equation*}
    g_t(u)
     = \frac{\beta \sinc (\beta u) - \alpha \sinc (\alpha u)}{\pi}  \; \exp \paren*{- \frac{u^2}{4t^2}}
   \end{equation*}
   where $\sinc(x) = \frac{\sin x}{x}$. 
  Therefore, the derivative of $g_t$ is
  \begin{equation*}
    g_t'(u) = - \frac{u}{2 t^2} \; g_t(u) + \frac{\beta^2 \sinc' (\beta u) - \alpha^2 \sinc' (\alpha u)}{\pi} \; \exp \paren*{- \frac{u^2}{4t^2}}.
  \end{equation*}
  We use the fact that $\abs{\sinc (x)} \leq 1$ and $\abs{x \sinc ' (x)} = \abs{\cos x - \sinc x} \leq 2$  to
  conclude.
\end{proof}

This leads directly to an estimate on the kernel function, by cutting the integral~\eqref{eq:selberg_transform}
expressing $K_t$ in terms of $g_t$ at $2 \rho$ and using the same inequalities as before for the denominator.

\begin{lemm}
  \label{lemm:estimate_K_H}
  Let $\frac 1 4 \leq a \leq b$ and $r \in (0,\rmax)$. For any $\rho \geq r$, $t \geq \tminH$,
  \begin{equation}
    \label{eq:estimate_K_H}
    K_t(\rho) 
    = \O { \frac{t \sqrt b}{r^2} \; \exp \paren*{-\frac{\rho^2}{4t^2}}}.
  \end{equation}
\end{lemm}

Then, the same summation process leads to the following lemma.

\begin{lemm}
  \label{lemm:kernel_sum_H}
  Let $\frac 1 4 \leq a \leq b$ and $r \in (0,\rmax)$. Let $X = \faktor{\mathcal{H}}{\Gamma}$ be a compact hyperbolic
  surface of injectivity radius larger than $r$. For any $t \geq \tminH$, $L \geq 8 t^2$,
  \begin{equation}
    \label{eq:kernel_sum_H}
    \mathcal{R}_K(X,t,a,b) =
    \O{\frac{t^3 \sqrt b}{r^4} \brac*{\exp \paren*{-L}
        + \frac{\volhyp_X (X^-(L))}{\volhyp_X (X)} \; \exp \paren*{L} }}
  \end{equation}
  where $X^-(L)$ is the set of points in $X$ of radius of injectivity smaller
  than $L$.
\end{lemm}

We can then conclude using the geometric properties of random surfaces. 

\begin{prop}
  \label{prop:geometric_term_H}
  For any large enough $g$, any $\frac 1 4 \leq a \leq b$ and any hyperbolic surface $X \in \mathcal{A}_g$ defined in
  Corollary~\ref{coro:geometry}, if we set $t = \frac{\sqrt{\log g}}{4 \sqrt{3}}$, then
  \begin{equation}
    \label{eq:kernel_H}
    \mathcal{R}_K(X,t,a,b) = \O{\sqrt{\frac{b}{\log g}}}.
  \end{equation}
\end{prop}

\subsection{Small eigenvalues term, and proof of Theorem~\ref{theo:upper_bound} away from them}
\label{sec:upper_bound_H}

The behavior of the function $h_t$ is different on the imaginary and real axes. Noticeably, the function $h_t$ is
positive on the real axis, but it is not real valued on the imaginary axis. This will cause some of the inequalities from
the previous part to fail. Also, when $a$ is close to~$\frac 1 4$, the modulus of $h_t$ on the segment
$\left[ - \frac i 2, \frac i 2 \right]$ becomes too large, and the remainder we will obtain will be unsatisfactory. This
is the reason why this test function is only suitable for values of $a$ greater than $\frac 1 2$.

We shall now deal with the small eigenvalues, so that they do not
intervene anymore afterwards. 
\begin{lemm}
  \label{lemm:small_eig_H}
  Let $\frac 1 2 \leq a \leq b$. For any compact hyperbolic surface $X$ and any $t>0$,
  \begin{equation}
    \label{eq:trace_formula_estimate_H}
    \frac{1}{\volhyp_X(X)} \sum_{r_j \notin \mathbb R} (h_t(r_j) + h_t(-r_j))
    = \O{\frac{1}{t}}.
  \end{equation}
\end{lemm}

\begin{proof}
  Let $\frac 1 2 \leq \alpha \leq \beta$ such that $a = \frac 1 4 + \alpha^2$ and $b = \frac 1 4 + \beta^2$.
  If $r_j \notin \R$, then $r_j = i y_j$ with $y_j \in \left[ - \frac 1 2, \frac
  1 2 \right]$.
  By Lemma~\ref{lemm:trace_decrease_H}, 
  \begin{equation*}
    \abs{h_t(\pm r_j)}
    \leq \frac{1}{2 \sqrt \pi \; \alpha t} \exp \paren*{t^2 (y_j^2 - \alpha^2)}
    = \O{\frac 1 t} \quad \text{since } \alpha \geq \frac 1 2 \cdot
  \end{equation*}
  The number of such terms is $\leq 2g-2 = \O{\volhyp_X(X)}$ by~\cite{otal2009}.
\end{proof}

When we put together equation~\eqref{eq:trace_applied_H}, Proposition
\ref{prop:integral_H}, \ref{prop:geometric_term_H} and Lemma~\ref{lemm:small_eig_H}, we obtain directly
the following statement.
\begin{coro}
  \label{coro:trace_formula_estimate_H}
  For any large enough $g$, any $\frac 1 2 \leq a \leq b$ and any hyperbolic surface $X \in \mathcal{A}_g$ defined in
  Corollary~\ref{coro:geometry}, if we set $t = \frac{\sqrt{\log g}}{4 \sqrt{3}}$, then
  \begin{equation}
    \label{eq:trace_formula_estimate_H}
    \frac{1}{\volhyp_X(X)} \sum_{r_j \in \R} (h_t(r_j) + h_t(-r_j))
    = \integraltermH + \O{\sqrt{\frac{b}{\log g}}}.
  \end{equation}
\end{coro}

It is straightforward to deduce Theorem~\ref{theo:upper_bound} from this result as was done before.
 
\subsection{Proof of Theorem \ref{theo:equivalent} away from small eigenvalues}
\label{sec:comparison_H}

A version of Theorem~\ref{theo:equivalent} in terms of $\alpha$ and $\beta$ follows directly from the method of
Section~\ref{sec:B}.

\begin{theo}[Theorem~\ref{theo:equivalent} away from small eigenvalues]
  \label{theo:equivalent_H}
  There exists a universal constant $C>0$ such that, for any large enough $g$, any
  $\frac 1 2 \leq \alpha \leq \beta$ and any hyperbolic surface
  $X \in \mathcal{A}_g$ from Corollary \ref{coro:geometry}, 
  if we set $a = \frac 1 4 + \alpha^2$ and $b = \frac 1 4 + \beta^2$, then one can write the counting function
  $\counting{\Delta}{X}{a,b}$ as
  \begin{equation}
    \label{eq:equivalent_H}
    \frac{\counting{\Delta}{X}{a,b}}{\volhyp_X (X)}
    = \integraltermH + R(X,a,b)
  \end{equation}
  where
  \begin{equation*}
    - C \sqrt{\frac{b}{\log g}} \leq R(X,a,b) \leq C \sqrt{\frac{b}{\log g}} \; \log \paren*{2 + (\beta -
      \alpha) \sqrt{\log g}}^{\frac 1 2}.
  \end{equation*}
\end{theo}

We translate this statement in terms of $a$ and $b$ thanks to the fact that
$\beta - \alpha = \frac{b-a}{\beta + \alpha}$, and therefore, as soon as $b \geq \frac 1 2$,
\begin{equation*}
  \beta - \alpha \leq \frac{b-a}{\sqrt{b - \frac 1 4}} \leq
  \sqrt 2 \; \frac{b-a}{\sqrt b} \cdot
\end{equation*}

\section{Proof of Corollary \ref{coro:jth}}
\label{sec:coro}

Let us use Theorem \ref{theo:equivalent} in order to estimate the $j$-th eigenvalue of a typical compact hyperbolic
surface. We recall that our aim is to prove that for a typical surface,
\begin{equation*}
  \lambda_j(X) = \frac j g + \O{1+ \sqrt{\frac j g \log \paren*{2+ \frac j g}}}.
\end{equation*}

\begin{proof}
  Let $g$ be large enough for Theorems \ref{theo:upper_bound} and \ref{theo:equivalent} to apply, and
  $X \in \mathcal{A}_g$. Let $j \geq 0$.

  If $\lambda_j(X) \leq \frac 1 4$, then  $j \leq 2g-2$ by work of Otal and Rosas~\cite{otal2009}. It follows that both $\lambda_j(X)$ and
  $\frac j g$ are $\O{1}$, which leads to our claim.

  We can therefore assume $\lambda_j(X) \geq \frac 1 4$. By Theorem \ref{theo:equivalent} applied between $0$ and
  $\lambda_j(X)$,
  \begin{align*}
    \frac{\counting{\Delta}{X}{0,\lambda_j(X)}}{2\pi(2g-2)}
    & = \frac{1}{4 \pi} \int_{\frac 1 4}^{\lambda_j(X)}
      \1{[a,b]}(\lambda) \tanh \paren*{\pi \sqrt{\lambda - \frac 1 4}} \d \lambda 
    + \O{\sqrt{\lambda_j(X) \log \paren*{2 + \lambda_j(X)}}} \\
    & = \frac{\lambda_j(X)}{4 \pi} + \O{1 + \sqrt{\lambda_j(X) \log (2+\lambda_j(X))}}.
  \end{align*}
  But by definition of the $j$-th eigenvalue $\lambda_j(X)$, we also have
  \begin{equation*}
    \counting{\Delta}{X}{0, \lambda_j(X)} = j + \O{m_X(\lambda_j(X))},
  \end{equation*}
  which is $j + \O{g \sqrt{\lambda_j(X)}}$ by Corollary \ref{coro:multiplicity}. As a consequence, there is a constant
  $C>0$ such that
  \begin{equation}
    \label{eq:lambdaj_minus_j_g}
    \abs*{\lambda_j(X) - \frac j g} \leq C \paren*{1 + \sqrt{\lambda_j(X) \log (2 + \lambda_j(X))}}.
  \end{equation}

  There exists a constant $M>0$ such that, as soon as $\lambda_j(X) > M$, the right hand term of equation \eqref{eq:lambdaj_minus_j_g}
  is smaller than $\frac{\lambda_j(X)}{2}$.
  We distinguish two cases.
  \begin{itemize}
  \item If $\lambda_j(X) > M$, then by equation \eqref{eq:lambdaj_minus_j_g} and by definition of $M$, $\lambda_j(X) \leq
    2\frac j g$. Thererefore, equation \eqref{eq:lambdaj_minus_j_g} leads to our claim.
  \item Otherwise, by equation \eqref{eq:lambdaj_minus_j_g}, $\frac j g$ and $\lambda_j(X)$ are both $\mathcal{O}_M(1)$, and the
    conclusion still follows.
  \end{itemize}
\end{proof}

\bibliographystyle{plain}
\bibliography{bibliography}

\end{document}